\documentclass[a4paper]{amsart}
\usepackage{amssymb}
\usepackage{xypic}
\usepackage{graphics}
\xyoption{poly}
\xyoption{curve}

\newtheorem{thm}{Theorem}[section]
\newtheorem*{thm*}{Theorem}
\newtheorem{prop}[thm]{Proposition}
\newtheorem{cor}[thm]{Corollary}
\newtheorem{lem}[thm]{Lemma}

\theoremstyle{definition}
\newtheorem{dfn}[thm]{Definition}
\newtheorem*{dfn*}{Definition}
\newtheorem{ex}[thm]{Example}
\newtheorem{exs}[thm]{Examples}

\theoremstyle{remark}
\newtheorem{rem}[thm]{Remark}

\numberwithin{equation}{section}

\newcommand{\RR}{\mathbb R}
\newcommand{\ZZ}{\mathbb Z}

\newcommand{\CC}{\mathbb C}

\newcommand{\im}{\mathrm{im}\,}

\newcommand{\supp}{\operatorname{supp}}
\newcommand{\mfk}{\mathfrak{k}}

\newcommand{\mfh}{\mathfrak{h}}
\newcommand{\mft}{\mathfrak{t}}
\newcommand{\mfp}{\mathfrak{p}}
\newcommand{\mfm}{\mathfrak{m}}
\newcommand{\Hb}{H_{bas}}
\newcommand{\Mr}{M_{reg}}
\newcommand{\Ms}{M_{sing}}

\newcommand{\depth}{\operatorname{depth}}
\newcommand{\grade}{\operatorname{grade}}
\newcommand{\Ann}{\mathrm{Ann}}
\newcommand{\Ass}{\mathrm{Ass}}

\begin{document}

\title{Torus actions whose equivariant cohomology is Cohen-Macaulay}
\author{Oliver Goertsches}
\address{Oliver Goertsches, Mathematisches Institut, Universit\"at zu K\"oln, Weyertal 86-90, 50931 K\"oln, Germany}
\email{ogoertsc@math.uni-koeln.de}
\author{Dirk T\"oben}
\address{Dirk T\"oben, Mathematisches Institut, Universit\"at zu K\"oln, Weyertal 86-90, 50931 K\"oln, Germany}
\email{dtoeben@math.uni-koeln.de}
\thanks{DT was supported by the Schwerpunktprogramm SPP 1154 of the DFG}
\begin{abstract} We study Cohen-Macaulay actions, a class of torus actions on manifolds, possibly without fixed points, which generalizes and has analogous properties as equivariantly formal actions. Their equivariant cohomology algebras are computable in the sense that a Chang-Skjelbred Lemma, and its stronger version, the exactness of an Atiyah-Bredon sequence, hold. The main difference is that the fixed point set is replaced by the union of lowest dimensional orbits. We find sufficient conditions for the Cohen-Macaulay property such as the existence of an invariant Morse-Bott function whose critical set is the union of lowest dimensional orbits, or open-face-acyclicity of the orbit space. Specializing to the case of torus manifolds, i.e., $2r$-dimensional orientable compact manifolds acted on by $r$-dimensional tori, the latter is similar to a result of Masuda and Panov, and the converse of the result of Bredon that equivariantly formal torus manifolds are open-face-acyclic. 
\end{abstract}
\maketitle
\tableofcontents

\section{Introduction}

In the theory of equivariant cohomology, the class of equivariantly formal actions of a real torus $T$ on a compact manifold $M$ is certainly one of the most intensely studied. On the one hand, it comprises many important examples, such as Hamiltonian torus actions on compact symplectic manifolds, and on the other hand such actions have many beautiful properties, e.g.~the equivariant cohomology $H^*_T(M)$ is determined by the 1-skeleton of the action as proven by Chang and Skjelbred \cite[Lemma 2.3]{ChangSkjelbred}, and thereby explicitly computable \cite[Theorem 1.2.2]{GKM} via what is nowadays called GKM theory, see e.g.~\cite{GuilleminZara}. To our knowledge, the only known big classes of actions on manifolds for which the ($S(\mft^*)$-algebra structure of the) equivariant cohomology is explicitly computable are either equivariantly formal or have only one isotropy type.

A geometric property of equivariantly formal actions is that their minimal strata consist of fixed points. 
The fixed point set of an action plays an important role in the whole theory, as shown by e.g.~the famous localization theorems. A basic example of its relevance is that it encodes the rank of $H^*_T(M)$ as a module over $S(\mft^*)$. 

The motivating question for our work was the following: {\it Is there a suitable generalization of equivariant formality that also covers actions without fixed points? }

From the point of view of computability of $H^*_T(M)$, the answer to this question is implicit in the proof of the exactness of the so-called Atiyah-Bredon sequence \cite{Bredon} (Theorem \ref{thm:atiyahbredonsequence}, see also \cite{FranzPuppe2003}), which can be regarded as an extension of the Chang-Skjelbred Lemma; the relevant property of $H^*_T(M)$ for the proof is not that it is a free module, but that it is a Cohen-Macaulay module of Krull dimension $\dim T$.

In Section \ref{sec:CMactions} we give the
\begin{dfn*} The $T$-action is {\it Cohen-Macaulay} if $H^*_T(M)$ is a Cohen-Macaulay module over $S(\mft^*)$.
\end{dfn*}
The Cohen-Macaulay property was used as a tool in various papers on equivariant cohomology. The purpose of this work is to justify why it is the appropriate notion for answering the above question. The central role of the fixed point set will be seen to be assumed by the union of lowest dimensional orbits. 

Let $b$ be the smallest occuring orbit dimension. Similarly to the equivariantly formal case \cite{Bredon, FranzPuppe}, there is an Atiyah-Bredon sequence, but with the fixed point set replaced by the union $M_b$ of $b$-dimensional orbits, whose exactness is equivalent to Cohen-Macaulayness (Theorem \ref{thm:genABsequence}):
\begin{thm*} The $T$-action on $M$ is Cohen-Macaulay if and only if the sequence
\[
0\to H^*_T(M)\to H^*_T(M_b)\to H^*_T(M_{b+1},M_b)\to \ldots \to H^*_T(M,\Ms)\to 0
\]
is exact.
\end{thm*}
The notion of Cohen-Macaulay action encompasses both equivariantly formal actions and actions with only one isotropy type, including locally free actions. Namely, if the action is equivariantly  formal, i.e., $H^*_T(M)$ is a free module, $H^*_T(M)$ is Cohen-Macaulay of 
maximal Krull dimension $\dim T$, and if the action is locally free it is Cohen-Macaulay of Krull dimension $0$. Effective cohomogeneity-one actions are examples of Cohen-Macaulay actions of Krull dimension $1$, see Example \ref{ex:CM}.
To show that there are many more interesting examples of Cohen-Macaulay actions, we relax well-known sufficient conditions for equivariant formality in a way that we retain the Cohen-Macaulay property: (Theorem \ref{thm:MorseBott})
\begin{thm*} If the action admits a $T$-invariant Morse-Bott function with critical set $M_b$, then the action is Cohen-Macaulay.
\end{thm*}
and (Theorem \ref{thm:diskbundle})
\begin{thm*} If the $T$-manifold $M$ admits a $T$-invariant disk bundle decomposition satisfying the properties of Theorem \ref{thm:diskbundle}, then the action is Cohen-Macaulay.
\end{thm*}
In every case, the philosophy is to replace the fixed point set by the union of lowest dimensional orbits. 

The second half of the paper is dedicated to the question whether it is possible to give conditions on the orbit space of the action that imply that the action is Cohen-Macaulay.  In Section \ref{sec:bottomstratum}, we prove an algebraic characterization of injectivity of the restriction map $H^*_T(M)\to H^*_T(B)$, where $B$ is the bottom stratum of the action, i.e., the union of the minimal strata. This yields as a corollary the geometric statement (Theorem \ref{thm:OFA+injective})
\begin{thm*} If the orbit space of the $T$-action is almost open-face-acyclic, then $H^*_T(M)\to H^*_T(B)$ is injective.
\end{thm*}
Masuda and Panov \cite{MasudaPanov} proved that a torus manifold, i.e., an orientable compact $2\dim T$-dimensional manifold with fixed points, is equivariantly formal with respect to $\ZZ$-coefficients if and only if it is locally standard and its orbit space is closed-face-acyclic. In Section \ref{sec:ActionsWithAcyclicOrbitSpace} we show (Corollary \ref{cor:dimMdimTb}, Theorem \ref{thm:OFAisCM})
\begin{thm*} If $T$ acts effectively on an orientable compact manifold $M$ with open-face-acyclic orbit space, then $\dim M=2\dim T-b$ and the action is Cohen-Macaulay.
\end{thm*}
Here $b$ is the lowest occuring orbit dimension. To a large extent, our proof uses the ideas of Masuda and Panov; however, there are several differences: Most importantly, several special features such as the so-called canonical models \cite[Section 4.2]{MasudaPanov}, see Davis and Januszkiewicz \cite[Section 1.5]{Davis}, or the equivalent characterization of equivariant formality via $H^{odd}(M)=0$, are not available in our setting. Our chain of arguments instead relies on the fact that equivariant injectivity holds a priori by the results in Section \ref{sec:bottomstratum}, and various other modifications.

Our result states that actions on $2\dim T$-dimensional orientable compact manifolds with open-face-acyclic orbit space are equivariantly formal. The converse of this statement was proven by Bredon \cite[Corollary 3]{Bredon}; we therefore obtain (Theorem \ref{thm:OFAequivariantlyformal}):
\begin{thm*}
A $T$-action on an orientable compact manifold $M$ with $\dim M=2\dim T$ is equivariantly formal if and only if its orbit space is open-face-acyclic.
\end{thm*}

{\it Acknowledgements.} We are grateful to S\"onke Rollenske for various discussions on Cohen-Macaulay modules, and useful comments on a preliminary version of this paper.
We wish to thank Matthias Franz and Volker Puppe for bringing to our attention Proposition 5.1 of their paper \cite{FranzPuppe2003} and the important connection between Cohen-Macaulay and equivariantly formal actions described in Remark \ref{rem:reductiontoequivformal}. These two points were independently noticed by the referee, whom we thank for his careful reading resulting in several improvements of the paper.
\section{Preliminaries}

Throughout this paper, $T=(S^1)^r$ will denote an $r$-dimensional real torus. We will use several cohomology theories, always with $\RR$ as coefficient ring. All of them are defined for arbitrary ($T$-)spaces, but have an equivalent description in the case of a differentiable manifold (with a differentiable $T$-action). $H^*$ denotes singular (or deRham) cohomology. For a $T$-space $X$, $H^*_T(X)$ denotes equivariant cohomology of the $T$-action, i.e., the cohomology of the Borel construction $X_T=X\times_T ET$ \cite{Borel, Hsiang}. In the differentiable case this coincides with equivariant de Rham cohomology, see e.g.~\cite{GuilleminSternberg}. We will use the Cartan model $H^*_T(X)=H(S(\mft^*)\otimes \Omega(X)^T,d_T)$, where $S(\mft^*)$ is the symmetric algebra on the dual of the Lie algebra $\mft$ of $T$, $\Omega(X)^T$ consists of the $T$-invariant differential forms on $X$, and $d_T$ is the equivariant differential. 

Given a $T$-action on an arbitrary space $X$,  the basic cohomology $\Hb^*(X)$ will be understood as the singular cohomology of the quotient $X/T$. In the differentiable case this coincides with the usual definition of basic cohomology \cite[Theorem 30.36]{Michor}: a differential form $\omega$ is {\it basic} if it is horizontal, i.e., the contractions with the $T$-fundamental vector fields vanish, and $T$-invariant. Then $\Hb^*(X)=H(\Omega^*_{bas}(X),d)$, where $(\Omega^*_{bas}(X),d)$ is the complex of basic forms, with $d$ the restriction of the usual differential. 

The projection $X_T\to X/T$ naturally induces a map
\begin{equation} \label{eqn:basicinclusion}
\Hb^*(X)\to H^*_T(X),
\end{equation}
which in the differentiable case can also be seen to be induced by the inclusion of complexes $(\Omega^*_{bas}(X),d)\to (S(\mft^*)\otimes \Omega(X)^T,d_T)$. Note that in general \eqref{eqn:basicinclusion} is not injective, see e.g.~\cite[Example C.18]{GGK}.

There are also relative and compactly supported versions of equivariant and basic cohomology, where the latter will be denoted by an additional index, e.g.~$H^*_{T,c}$ for compactly supported equivariant cohomology \cite[Section 11.1]{GuilleminSternberg}.

In case $T$ acts locally freely on an orientable manifold $M$, then $M/T$ is an orientable orbifold and hence \cite{Satake} there is 
Poincar\'{e} duality for basic cohomology: $\Hb^*(M)=H^*(M/T)\cong H^{\dim M/T-*}_c(M/T)=H^{\dim M/T-*}_{bas,c}(M)$.

For a $T$-action on a manifold $M$ and any subspace $\mfk\subset \mft$, let $M^\mfk$ be the common zero set of all fundamental vector fields induced by $\mfk$. In other words, $M^\mfk$ is the fixed point set of the action of the connected Lie subgroup of $T$ with Lie algebra $\mfk$. For $p\in M$, let $M^p$ be the connected component of $M^{\mft_p}$ that contains $p$. A point is {\it regular} if its isotropy algebra is minimal among all isotropy algebras, and nonregular points are called {\it singular}. The regular set of the $T$-action on $M$ is denoted $\Mr$, and the singular set $\Ms$. The respective regular and singular sets of the $T$-action on $M^p$ are $\Mr^p=\{q\in M^p\mid \mft_q=\mft_p\}$ and $\Ms^p= M^p\setminus \Mr^p$. 
The $M^p$ are partially ordered by inclusion. The minimal elements $M^p$ with respect to this ordering are exactly those with $M^p=\Mr^p$. By definition, the {\it bottom stratum} $B$ of the action is the union of those minimal elements:
\[
B=\{p\in M\mid M^p=\Mr^p\}.
\]
(Strictly speaking, one should refer to $B$ as the infinitesimal bottom stratum, as for the definition of regularity of a point we only consider its isotropy algebra instead of its isotropy group.) The set of fixed points $M^\mft$ is contained in $B$.

For $i\geq 0$, let $M_i$ be the union of orbits of dimension $\leq i$. Clearly,
\[
M_i=\bigcup_{p\in M:\, \dim \mft_p\geq r-i} M^p.
\]
Furthermore, we define $M_{(i)}=M_i\setminus M_{i-1}$ to be the union of orbits of dimension equal to $i$. The $M_{(i)}$ are disjoint unions of submanifolds of type $\Mr^p$. In general, $M_i$ is not a submanifold of $M$, but still it is an equivariant strong deformation retract of some neighborhood in $M$ (or $M_{i+1}$). Of importance will be the cohomology of the pair $(M_i,M_{i-1})$. For later use, we note that 
\[
\Hb^*(M_i,M_{i-1})=H^*_{bas,c}(M_{(i)})=H^*_c(M_{(i)}/T)
\]
and
\[
H^*_T(M_i,M_{i-1})=H^*_{T,c}(M_{(i)}).
\]

\section{General lemmata}

In this section, we collect some more or less well-known lemmata on equivariant cohomology. If $A$ is a finitely generated graded $S(\mft^*)$-module, we define the {\it support} of $A$ as in \cite[Section 11.3]{GuilleminSternberg} to be the set of complex zeroes of the annihilator ideal of $A$:
\[
\supp A=\{x\in \mft\otimes \CC\mid f(x)=0 \text{ for all } f\in S(\mft^*) \text{ with } fA=0\}.
\]
Because $A$ is graded, the annihilator ideal is a graded ideal, and $\supp A$ is a conic subvariety of $\mft\otimes \CC$, i.e., for $x\in \supp A$ and $\lambda\in \CC$ we have $\lambda x\in \supp A$. For an element $\alpha\in A$, we define the support of $\alpha$ to be the support of the submodule of $A$ generated by $\alpha$: $\supp \alpha = \supp S(\mft^*)\alpha$. Clearly, for every $\alpha\in A$ we have $\supp \alpha\subset \supp A$.

\begin{lem} \label{lem:trivialpartofaction}
Let a torus $T$ act on a manifold $M$. If $K\subset T$ is any (not necessarily closed) Lie subgroup that acts trivially on $M$ and $K'\subset T$ a subtorus of $T$ such that $\mfk \oplus \mfk'=\mft$, then
\[
H^*_T(M)=S(\mfk^*)\otimes_{\RR} H^*_{K'}(M)
\]
as an $S(\mft^*)=S(\mfk^*)\otimes S(\mfk'^*)$-algebra.
\end{lem}
\begin{proof}
We have $S(\mft^*)\otimes \Omega(M)^T=S(\mfk^*)\otimes (S(\mfk'^*)\otimes \Omega(M)^{K'}).$
The equivariant differential $d_T$ leaves the space $S(\mfk'^*)\otimes \Omega(M)^{K'}$ invariant and is zero on $S(\mfk^*)$. Thus, $
H^*_T(M)=S(\mfk^*)\otimes H^*(S(\mfk'^*)\otimes \Omega(M)^{K'},d_{K'})=S(\mfk^*)\otimes H^*_{K'}(M).$
\end{proof}

\begin{lem}\label{lem:kernelM^K}
Let $T$ act on a compact manifold $M$ and  $\mfk\subset \mft$ be any subspace. Then the kernel of the restriction map $H^*_T(M)\to H^*_T(M^\mfk)$ is given by 
$\{\alpha\mid \mfk \otimes \CC \not\subset\supp\alpha\}$.
\end{lem}
\begin{proof}
Note that $M^\mfk$ is a closed $T$-invariant submanifold of $M$. By \cite[Theorem 11.4.2]{GuilleminSternberg} the kernel of the restriction map $i^*$ has support in the union of those subalgebras that occur as isotropy algebras in $M\setminus M^\mfk$, i.e.
\[
\supp\ker i^*\subset \bigcup_{\mfh \text{ isotropy algebra}, \mfk\not\subset\mfh} {\mfh \otimes \CC}.
\]
Thus, $\mfk\otimes \CC \not\subset \supp\ker i^*$. Since for any $\alpha\in \ker i^*$, we have $\supp \alpha\subset \supp \ker i^*$, we obtain $\mfk\otimes \CC \not\subset \supp \alpha$.

Conversely, let $\alpha\not\in \ker i^*$ and choose a subtorus $K'\subset T$ such that $\mft=\mfk\oplus \mfk'$. Then, $0\neq i^*\alpha\in H^*_T(M^\mfk)=S({\mathfrak k}^*)\otimes H^*_{K'}(M^\mfk) $ by the previous lemma, so $\mfk\otimes \CC\subset \supp i^*\alpha\subset \supp \alpha$. 
\end{proof}

\begin{prop}\label{prop:characterizationlocallyfree}
Let a torus $T$ act on a manifold $M$. Then the following properties are equivalent:
\begin{enumerate}
	\item The action is locally free.
	\item $H^k_T(M)=0$ for large $k$.
	\item $\supp H^*_T(M)=\{0\}$.
	\item The support of $1\in H^*_T(M)$ is $\{0\}$.
\end{enumerate}
\end{prop}
\begin{proof} We show $(1)\Rightarrow(2)\Rightarrow(3)\Rightarrow(4)\Rightarrow(1)$.
If the action is locally free, $H^*_T(M)=H^*(M/T)$  (this is standard for free actions, but also true  for only locally free actions because we are using real coefficients, see e.g.~Section C.2 of \cite{GGK}), so $(1)$ implies $(2)$. 

Assuming $(2)$, there is some $N>0$ such that any closed equivariant differential form of degree at least $N$ is exact. Thus, writing $S(\mft^*)=\RR[u_1,\ldots,u_n]$, it follows that for any $i$ some power of $u_i$ annihilates $H^*_T(M)$. Therefore the common zero set of the polynomials that annihilate $H^*_T(M)$ consists only of the element $0$, hence $(3)$. 

It is clear that $(3)$ implies $(4)$, so it remains to show that $(4)$ implies $(1)$. Clearly, the complexification of every isotropy algebra $\mfk$ is contained in the support of $1$, because $1$ is not in the kernel of $H^*_T(M)\to H^*_T(M^\mfk)$. Thus, if the support of $1$ is $\{0\}$, the action is locally free.
\end{proof}

\section{Equivariant formality}\label{sec:EquivariantFormality}

The $T$-action on $M$ is equivariantly formal in the sense of \cite{GKM} if the cohomology spectral sequence associated with the fibration $M\times_T ET\to BT$ collapses at $E_2$. The following are well-known equivalent characterizations of equivariant formality:
\begin{enumerate}
\item $H^*_T(M)$ is free as an $S(\mft^*)$-module \cite[Corollary 4.2.3]{AlldayPuppe}.
\item $\dim H^*(M^T)=\dim H^*(M)$ \cite[Corollary IV.2]{Hsiang}. The inequality $\leq$ is true for any $T$-action.
\end{enumerate}
An important sufficient condition for equivariant formality is $H^{odd}(M)=0$, see \cite[Theorem 6.5.3]{GuilleminSternberg}.

Also the following proposition is fairly standard. One way to prove it is as an application of characterization (2) above, see e.g.~the proof of the Main Lemma in \cite{Bredon}; a proof using (1) is given in \cite[Lemma 2]{Brion}. 
\begin{prop}\label{prop:eqformalM^K}
If the $T$-action on $M$ is equivariantly formal, then for any subtorus $K\subset T$, the $T/K$-action on $M^\mfk$ is equivariantly formal.
\end{prop}

The next lemma can be proven as a direct application of characterization (2) above.
\begin{lem} \label{lem:eqformalcomponent}
A $T$-action on $M$ is equivariantly formal if and only if the $T$-action on each connected component of $M$ is equivariantly formal.
\end{lem}

The following corollary appears (with a different proof) as Theorem 11.6.1 in \cite{GuilleminSternberg}, and as Proposition C.28 in \cite{GGK}.
\begin{cor} \label{cor:bottomstrofequivformal} If the $T$-action on $M$ is equivariantly formal and $K\subset M$ is any subtorus, each connected component of $M^\mfk$ contains a $T$-fixed point. In other words: the bottom stratum of an equivariantly formal action is equal to the set of fixed points.
\end{cor}
\begin{proof}
This follows from Proposition \ref{prop:eqformalM^K} and Lemma \ref{lem:eqformalcomponent} because every equivariantly formal action has a fixed point. 
\end{proof}
The relevance of the notion of equivariant formality emerges from the fact that the equivariant cohomology of spaces satisfying this condition is (relatively) easy to compute, thanks to the exact sequence
\[
0\to H^*_T(M)\to H^*_T(M_0) \overset{\partial}{\to} H^*_T(M_1,M_0).
\]
Here, $\partial$ is the boundary operator in the long exact sequence of the pair $(M_1,M_0)$. Injectivity of $H^*_T(M)\to H^*_T(M_0)$ follows because the kernel of this map is the module of torsion elements \cite[Theorem 11.4.4]{GuilleminSternberg} and $H^*_T(M)$ is torsion-free\footnote{This condition is not equivalent to equivariant formality; see \cite{FranzPuppe2008} for an example.}. Exactness at $H^*_T(M_0)$ is the so-called Chang-Skjelbred Lemma \cite[Lemma 2.3]{ChangSkjelbred}, see also \cite[Section 11.5]{GuilleminSternberg}.

The following characterization of equivariant formality is an extension of the exact sequence above.
\begin{thm}[\cite{Bredon}, \cite{FranzPuppe}]\label{thm:atiyahbredonsequence} The $T$-action on $M$ is equivariantly formal if and only if the sequence
$$0\to H^*_T(M)\to H^*_T(M_0)\to H^*_T(M_1,M_0)\to \ldots \to H^*_T(M,\Ms)\to 0$$
is exact.
\end{thm}
In this sequence, the maps $H^*_T(M_i,M_{i-1})\to H^*_T(M_{i+1},M_{i})$ are the boundary operators of the triples $(M_{i+1},M_i,M_{i-1})$. Exactness of the sequence under the condition of equivariant formality was proven by Bredon \cite[Main Lemma]{Bredon}, adapting an analogous result of Atiyah in equivariant $K$-theory \cite[Lecture 7]{AtiyahSeq}. Following Franz and Puppe, we will therefore refer to this sequence as the Atiyah-Bredon sequence. The converse direction is due to Franz and Puppe \cite[Theorem 1.1]{FranzPuppe}. Note that because we are using real coefficients we do not need any assumptions on the connectedness of isotropy groups, as is pointed out in Remark 1.2 of \cite{FranzPuppe}. For a version of the result of Atiyah and Bredon for other coefficients, see \cite{FranzPuppe2003}.

The following proposition shows that exactness of the Atiyah-Bredon sequence at $H^*_T(M)$ and $H^*_T(M_0)$ is implied by exactness of the rest of the sequence.
\begin{prop} \label{prop:vorneexakt}
Let $H^*$ denote either equivariant or basic cohomology. If for some $i$, the truncated Atiyah-Bredon sequence 
$$H^*(M_i)\to H^*(M_{i+1},M_i)\to \ldots \to H^*(M,\Ms)\to 0$$
is exact, then also
$$0\to H^*(M)\to H^*(M_i)\to H^*(M_{i+1},M_i)\to \ldots \to H^*(M,\Ms)\to 0.$$
\end{prop}
One way to see this is to regard the maps in the sequence as differentials in the spectral sequence associated to the filtration $M_i\subset M_{i+1}\subset \ldots \subset M_{r-1}\subset M_r$. The assumption implies that this spectral sequence collapses, and because it converges to $H^*(M)$, the statement follows. This argument was used in \cite[p.~846]{Bredon}. It can also be proven by hand: for the injectivity of $H^*(M)\to H^*(M_i)$, a straightforward diagram chase shows  that that for every $j\geq i$, we have $\ker (H^*(M)\to H^*(M_j))=\ker (H^*(M)\to H^*(M_{j+1}))$. For exactness at $H^*(M_i)$, another diagram chase proves that for every $j\geq i+1$, the image of $H^*(M_j)\to H^*(M_i)$ equals the image of $H^*(M_{j+1})\to H^*(M_i)$.

\section{Cohen-Macaulay modules over graded rings}

In the literature, the Cohen-Macaulay property usually is considered for modules over Noetherian (local) rings. In our situation, it is natural to consider the graded version of this concept. 

Using the language of e.g.~\cite[Section 1.5]{BrunsHerzog}, a graded ring $R$ (graded over the integers) is *local if it has a unique *maximal ideal, where a *maximal ideal is a graded ideal $\mfm\neq R$ which is maximal among the graded ideals. Thus, $S(\mft^*)$ is a Noetherian graded *local ring. Note that in general a *maximal ideal is not necessarily maximal.

Let $R$ be a Noetherian graded *local ring, with *maximal ideal $\mfm$. Then the {\it depth} of a finitely generated graded module $A$ over $R$ is defined as the length of a maximal $A$-regular sequence in $\mfm$:
\[
\depth A=\grade(\mfm,A).
\]
The {\it Krull dimension} of $A$, denoted $\dim A$, is defined as the Krull dimension of the ring $R/\Ann(A)$, where $\Ann(A)=\{r\in R\mid rA=0\}$, i.e., the supremum of the lengths of chains of prime ideals in $R$ containing $\Ann(A)$. 

\begin{dfn} A finitely generated graded module $A$ over a Noetherian graded *local ring $R$ is  {\it Cohen-Macaulay} if $\depth A=\dim A$.
\end{dfn}

Instead of working with the graded notions, we could equally well localize everything at the *maximal ideal (as e.g.~Franz and Puppe \cite{FranzPuppe2003} do it in their proof of the exactness of the Atiyah-Bredon sequence for other coefficients), because of the following proposition:
\begin{prop} Let $A$ be a finitely generated graded module over a Noetherian graded *local ring $R$ with *maximal ideal $\mfm$ such that $R/\mfm$ is a field (e.g.~$R=S(\mft^*)$). Then the following conditions are equivalent:
\begin{enumerate}
\item $A$ is Cohen-Macaulay over $R$ 
\item $A_\mfm$ is Cohen-Macaulay over the local ring $R_\mfm$
\item $A_{\mfm'}$ is Cohen-Macaulay over the local ring $R_{\mfm'}$ for all (not necessarily graded) maximal ideals $\mfm'\subset R$.
\end{enumerate}
If these conditions are satisfied, then the Krull dimensions of the $R$-module $A$ and the $R_\mfm$-module $A_\mfm$ coincide.
\end{prop}
\begin{proof}
The equivalence of $(2)$ and $(3)$ is \cite[Exercise 2.1.27.(c)]{BrunsHerzog} and valid without the additional assumption that $R/\mfm$ is a field; note that there, Cohen-Macaulay modules (over arbitrary Noetherian rings) are defined via condition $(3)$ of this proposition. 

We only explain the equivalence of $(1)$ and $(2)$. In fact, we show that the depths of $A$ and $A_\mfm$ and the Krull dimensions of $A$ and $A_\mfm$ coincide without using the Cohen-Macaulay property. We have $\depth A=\grade(\mfm,A)=\depth A_\mfm$ by \cite[Prop.~1.5.15.(e)]{BrunsHerzog}. For the equality of the dimensions, note that $\dim A=\sup_{\mfm'} \dim A_{\mfm'}$, where $\mfm'$ varies over all maximal ideals in $R$, so obviously $\dim A_\mfm\leq \dim A$. For the other inequality we use that if $\mfm'\neq \mfm$ is another maximal ideal, then by \cite[Theorem 1.5.8.(b)]{BrunsHerzog} the largest graded prime ideal $\mfp\subset \mfm'$ satisfies $\dim A_{\mfm'}=\dim A_\mfp +1$. Because of the assumption that $R/\mfm$ is a field, $\mfp$ is properly contained in the *maximal ideal $\mfm$, so $\dim  A_\mfp<\dim A_\mfm$, and hence $\dim A_{\mfm'}\leq \dim A_\mfm$.
\end{proof}
The proof implies that for a finitely generated module over an arbitrary Noetherian graded *local ring $R$, the inequality $\depth A_\mfm\leq \dim A_\mfm$ for the localized module $A_\mfm$ over $R_\mfm$ translates to the corresponding one for $A$: if $A\neq 0$, then
\[
\depth A\leq \dim A.
\]

For later use, we collect some well-known lemmata on Cohen-Macaulay modules. The first two, which describe how depth and the Cohen-Macaulay property behave with respect to short exact sequences, will be crucial for all our proofs that the equivariant cohomology of actions with certain properties is Cohen-Macaulay.
\begin{lem}[{\cite[Proposition 1.2.9]{BrunsHerzog}}] \label{lem:depthses} Let $0\to A\to B\to C\to 0$ be an exact sequence of finitely generated graded modules over a Noetherian graded *local ring $R$. Then
\begin{enumerate}
\item $\depth A\geq \min \{\depth B,\depth C+1\}$
\item $\depth B\geq \min\{\depth A,\depth C\}$
\item $\depth C\geq \min \{\depth A-1,\depth B\}$
\end{enumerate}
\end{lem}
\begin{lem}\label{lem:CMseq}
Let $0\to A\to B\to C\to 0$ be an exact sequence of finitely generated graded modules over a Noetherian graded *local ring $R$. Then the following statements are true:
\begin{enumerate}
\item  If $A$ and $C$ are Cohen-Macaulay of the same Krull dimension $n$, then $B$ is also Cohen-Macaulay of Krull dimension $n$.
\item If $B$ and $C$ are Cohen-Macaulay of the same Krull dimension $n$, then either $A=0$ or $A$ is also Cohen-Macaulay of Krull dimension $n$.
\end{enumerate}
\end{lem}
\begin{proof}
We have $\dim B=\sup_{\mfp} \dim R/\mfp$, where the supremum is taken over those prime ideals $\mfp$ with $B_\mfp\neq 0$, see \cite[Ch.~III, B.1]{Serre}. Thus, Proposition I.4 of \cite{Serre} implies that $\dim B=\max\{\dim A,\dim C\}$ for any short exact sequence of finitely generated $R$-modules. To see how the depths of the modules relate, use Lemma \ref{lem:depthses}: in case $(1)$, it implies that $n\leq \depth B\leq \dim B = n$, since $\depth B\leq \dim B$ in any case. In case $(2)$ and $A\neq 0$, the lemma implies that $n\leq \depth A\leq \dim A\leq n$.
\end{proof}
\begin{rem} If $A$ and $B$ are Cohen-Macaulay of the same Krull dimension $n$ and $C\neq 0$, then $C$ is not necessarily Cohen-Macaulay of Krull dimension $n$. For example, consider the short exact sequence $0\to \RR[t]\overset{\cdot t}{\to}\RR[t] \to \RR\to 0$.
\end{rem}

\begin{ex}\label{ex:MiMi-1CM}
If a $T$-action on a compact manifold $M$ and $i$ is such that $M_{i-1}\neq M_i$, then $H^*_T(M_i,M_{i-1})$ is a Cohen-Macaulay module of Krull dimension $r-i$. In fact, if we choose points $p_j$ such that $M_{(i)}$ is the disjoint union of the $\Mr^{p_j}$, then
\[
H^*_T(M_i,M_{i-1})=H^*_{T,c}(M_{(i)})=\bigoplus_j H^*_{T,c}(\Mr^{p_j}) = \bigoplus_j H^*_{bas,c}(\Mr^{p_j})\otimes S(\mft_j^*),
\]
where $\mft_j$ is the unique isotropy algebra of $\Mr^{p_j}$. Thus, $H^*_T(M_i,M_{i-1})$ is the sum of Cohen-Macaulay modules of Krull dimension $r-i$, and Lemma \ref{lem:CMseq} implies that it is Cohen-Macaulay of Krull dimension $r-i$ itself.
\end{ex}

\begin{lem}[{\cite[Lemma 4.3]{FranzPuppe2003}}]\label{lem:submodulesCM}
Let $A$ be a finitely generated module over a Noetherian graded *local ring $R$. If $A$ is Cohen-Macaulay and $B\subset A$ a non-zero submodule, then $\dim B=\dim A$.
\end{lem}
\begin{proof} We only need to show $\dim B\geq \dim A$. Let $\mfp\in \Ass(B)$ be an associated prime ideal of $B$, i.e., $\mfp$ is the annihilator of some element in $B$. Because $B_\mfp\neq 0$, we have $\dim B\geq \dim R/\mfp$. Furthermore, Proposition 1.2.13.~of \cite{BrunsHerzog} is true in the graded setting, and hence $\dim R/\mfp \geq \depth A=\dim A$.
\end{proof}

\section{Cohen-Macaulay actions}\label{sec:CMactions}
In this section, we introduce our main object of study. Examining the proof that equivariant formality implies the exactness of the Atiyah-Bredon sequence \cite{Bredon} (see also \cite{FranzPuppe2003}), one sees that the relevant property of $H^*_T(M)$ is not that it is a free $S(\mft^*)$-module, but that it is a Cohen-Macaulay module of Krull dimension $r=\dim T$.

It is proven in \cite[Proposition 5.1]{FranzPuppe2003} that for any $T$-action on $M$ the Krull dimension of the $S(\mft^*)$-module $H^*_T(M)$ equals the dimension of a maximal isotropy algebra (i.e., the Lie algebra of an isotropy group).

\begin{dfn}
Let $T$ act on a compact manifold $M$. We say that the action is {\it Cohen-Macaulay} if $H^*_T(M)$ is a Cohen-Macaulay module over $S(\mft^*)$.
\end{dfn}
If $b$ denotes the lowest occuring dimension of a $T$-orbit, i.e., $M_b\neq \emptyset$ but $M_{b-1}=\emptyset$, then the dimension of a maximal isotropy algebra is $\dim T-b=r-b$. 
With this definition, Theorem \ref{thm:atiyahbredonsequence} is still valid in the sense of the following Theorem. 
\begin{thm} \label{thm:genABsequence} Let $T$ act on a compact manifold $M$, and denote by $b$ the lowest occuring dimension of a $T$-orbit. Then the following conditions are equivalent:
\begin{enumerate}
\item  The action is Cohen-Macaulay
\item The sequence
\[
0\to H^*_T(M)\to H^*_T(M_b)\to H^*_T(M_{b+1},M_b)\to \ldots \to H^*_T(M,\Ms)\to 0
\]
(which we will refer to as the Atiyah-Bredon sequence) is exact.
\end{enumerate}
\end{thm}
\begin{proof} For the proof of $(1)\Rightarrow (2)$, we follow \cite{FranzPuppe2003} closely. The proof of $(2)\Rightarrow (1)$ naturally reverses the arguments of $(1)\Rightarrow (2)$.

Without loss of generality, we can assume that the action is effective, i.e., $M=M_r$ and $\Ms=M_{r-1}$. For both directions we will use the following consequence of the localization theorem, see \cite[Lemma 4.4]{FranzPuppe2003},
\begin{equation} \label{eq:dimMMj}
\dim H^*_T(M,M_j)\leq r-j-1
\end{equation}
and the fact that exactness of the Atiyah-Bredon sequence is equivalent to the exactness of the sequences
\begin{equation}\label{eqn:atiyahbredonshort}
0\to H^*_T(M,M_{j-1})\to H^*_T(M_{j},M_{j-1})\to H^*_T(M,M_{j})\to 0
\end{equation}
for $j\geq b$, see \cite[Lemma 4.1]{FranzPuppe}.

Assuming the action is Cohen-Macaulay, we prove by induction that \eqref{eqn:atiyahbredonshort} is exact and that $H^*_T(M,M_j)$ is Cohen-Macaulay of Krull dimension $r-j-1$. Assume we have shown that $H^*_T(M,M_{j-1})$ is Cohen-Macaulay of Krull dimension $r-j$. For the exactness of \eqref{eqn:atiyahbredonshort}, we show that $H^*_T(M,M_j)\to H^*_T(M,M_{j-1})$ is the zero map. Since $H^*_T(M,M_{j-1})$ is Cohen-Macaulay of Krull dimension $r-j$, the image of $H^*_T(M,M_j)$ under the map in question is, by Lemma \ref{lem:submodulesCM}, either zero or has Krull dimension $r-j$ as well. But on the other hand, \eqref{eq:dimMMj} implies that its Krull dimension is at most $r-j-1$, hence the image vanishes and \eqref{eqn:atiyahbredonshort} is exact. Because $\depth H^*_T(M_j,M_{j-1})=r-j$ by Example \ref{ex:MiMi-1CM}, Lemma \ref{lem:depthses} implies that $\depth H^*_T(M,M_j)\geq r-j-1$. Noting that $H^*_T(M,M_j)\neq 0$, this shows together with \eqref{eq:dimMMj} that $H^*_T(M,M_j)$ is Cohen-Macaulay of Krull dimension $r-j-1$. 

For the other direction, assume that the sequences \eqref{eqn:atiyahbredonshort} are exact. By Example \ref{ex:MiMi-1CM}, $H^*_T(M,M_{r-1})$ is a Cohen-Macaulay module of Krull dimension $0$.  Using $\depth H^*_T(M_j,M_{j-1})=r-j$, Lemma \ref{lem:depthses} implies that if $H^*_T(M,M_j)$ is a Cohen-Macaulay module of Krull dimension $r-j-1$, then $\depth H^*_T(M,M_{j-1})\geq r-j$ and hence \eqref{eq:dimMMj} implies that $H^*_T(M,M_j)$ is Cohen-Macaulay of Krull dimension $r-j$. By induction it follows that $H^*_T(M)$ is Cohen-Macaulay of Krull dimension $r-b$.
\end{proof}

\begin{rem} \label{rem:reductiontoequivformal} If the lowest occuring dimension of a $T$-orbit is $b$, then a generic $b$-dimensional subtorus $K\subset T$ acts locally freely. Choosing another subtorus $K'$ such that $\mfk\oplus \mfk'=\mft$, we have that $K'$ acts on $M/K$ with fixed points, and $H^*_T(M)\cong H^*_{K'}(M/K)$ as graded rings by the commuting action principle \cite[Section 4.6]{GuilleminSternberg}. Note that $M/K$ is not necessarily a manifold. The $T$-action on $M$ is Cohen-Macaulay if and only if the $K'$-action on $M/K$ is equivariantly formal. Note also that the $T$-action on $M$ and the $K'$-action on $M/K$ have the same orbit space. With this reduction one can alternatively deduce the theorem above from Theorem \ref{thm:atiyahbredonsequence} (respectively a version for more general spaces than differentiable manifolds).  Modulo the fact that $M/K$ is not a manifold, also the results in Section \ref{sec:MorseBott} and \ref{sec:diskbundle} (but not \ref{sec:bottomstratum} and \ref{sec:ActionsWithAcyclicOrbitSpace}) could be derived similarly. We believe however that is is illuminating to see that the existing proofs can easily be modified from the equivariantly formal to the Cohen-Macaulay setting.
\end{rem}

In the case $b=0$, i.e., $M_b=M_0=M^\mft$, the theorem reduces to Theorem \ref{thm:atiyahbredonsequence} by Atiyah, Bredon, Franz and Puppe because of
\begin{prop} If the $T$-action has fixed points, then it is Cohen-Macaulay if and only if it is equivariantly formal.
\end{prop} 
\begin{proof} This follows from the graded version of the Auslander-Buchs\-baum Theorem \cite[Exercise 19.8]{Eisenbud} and the fact that for a graded module over a polynomial ring, the projective dimension is equal to the length of the minimal free resolution \cite[Corollary 1.8]{Syzygies}.
\end{proof}
Note that in the proof of $(2)\Rightarrow (1)$ in Theorem \ref{thm:atiyahbredonsequence} by Franz and Puppe \cite{FranzPuppe} this argument is not needed, as they directly show freeness of $H^*_T(M)$.

\begin{exs}\label{ex:CM}
\begin{enumerate} \item Let $T$ act on $M$, and denote by $K$ the connected component of the kernel of the action, i.e., of the subgroup of $T$ consisting of the elements that act trivially. Then the $T$-action on $M$ is Cohen-Macaulay if and only if the $T/K$-action on $M$ is Cohen-Macaulay.
\item $T$-actions with only one local isotropy type are Cohen-Macaulay. The Atiyah-Bredon sequence in this case is $0\to H^*_T(M)\to H^*_T(M)\to 0$.
\item Let $T$ act on $M$ effectively and with cohomogeneity one, i.e., the regular orbits have codimension one. There are at most two singular orbits, all of which have codimension two (recall that for us only the local isotropy type is relevant). If the action is locally free, then the action is Cohen-Macaulay by the previous example, so we can assume there exists at least one singular orbit. In this case, the singular orbits have dimension $r-1$ because the regular orbits are $S^1$-fibre bundles over the singular orbits, i.e., $b=r-1$. To see that the Atiyah-Bredon sequence in this case is exact, we need to show that $H^*_T(M,M_{r-1})\to H^*_T(M)$ is the zero map. It is known that the orbit space is homeomorphic to the closed interval $[-1,1]$, and the orbit space of the regular stratum is either the open interval $(-1,1)$ or a half-open interval $[-1,1)$, where the latter case can only occur in the non-orientable case. In any case, $H^k_T(M,M_{r-1})=H^k_c(\Mr/T)=0$ for $k\neq 1$. But $H^1_T(M)=H^1(M/T)=H^1([-1,1])=0$ by \cite[Example C.8]{GGK}, which
shows that the map in question is the zero map. Thus, cohomogeneity-one actions are Cohen-Macaulay.
\end{enumerate}
\end{exs}

The following proposition is a generalization of Corollary \ref{cor:bottomstrofequivformal}.
\begin{prop} \label{prop:M_bbottom}
If the action is Cohen-Macaulay, then the bottom stratum of the action equals $M_b$. More generally, this is true if $H^*_T(M)\to H^*_T(M_b)$ is injective.
\end{prop}
\begin{proof}
Obviously $M_b$ is contained in the bottom stratum. If there is a component $N$ of the bottom stratum not contained in $M_b$, then $N\cap M_b=\emptyset$. Then, the Thom class of $N$ (with respect to any chosen orientation on the normal bundle) is a nonzero class in $H^*_T(M)$ (because it restricts to the Euler class of $N$ which is nonzero by \cite[Proposition 3]{Duflot}) that restricts to zero on $H^*_T(M_b)$.
\end{proof}

\begin{prop} \label{prop:kernelspaltet} If the action is Cohen-Macaulay, then for every $T$-invariant closed subspace $N$ containing $M_{b+1}$, we have $H^*_T(N)\cong H^*_T(M)\oplus (\ker H^*_T(N)\to H^*_T(M_b))$.
\end{prop}
\begin{proof} Because $H^*_T(M)\to H^*_T(M_b)$ is injective, $H^*_T(M)\to H^*_T(N)$ is injective as well, and its image does not intersect $\ker H^*_T(N)\to H^*_T(M_b)$. Exactness of the Atiyah-Bredon sequence at $H^*_T(M_b)$ implies that the image $ H^*_T(M)\to H^*_T(M_b)$ equals the image of $H^*_T(M_{b+1})\to H^*_T(M_b)$. Because $N$ is supposed to contain $M_{b+1}$, this subspace of $H^*_T(M_b)$ is also the same as the image of $H^*_T(N)\to H^*_T(M_b)$, whence the induced map $H^*_T(M)\to H^*_T(N)/(\ker H^*_T(N)\to H^*_T(M_b))$ is an isomorphism. 
\end{proof}

The following lemma is obvious.

\begin{lem} Assume the lowest occurring dimension of a $T$-orbit is the same for every connected component of $M$. Then the action is Cohen-Macaulay if and only if the action on each connected component is Cohen-Macaulay. 
\end{lem}

\section{Morse-Bott functions}\label{sec:MorseBott}

It is well-known that a $T$-action admitting a Morse-Bott function whose critical set is the fixed point set of the action is equivariantly formal. Replacing the fixed point set by the set of lowest dimensional orbits, we obtain the following generalization of this criterion:

\begin{thm}\label{thm:MorseBott}
Assume $T$ acts on a compact manifold $M$, and let $b$ be the dimension of the smallest occuring orbit. If there exists an invariant Morse-Bott function $f$ whose critical set is equal to a union of connected components of $M_b$, then the action is Cohen-Macaulay. Moreover, if $\kappa_0$ is the absolute minimum of $f$, then $H^*_T(M)\to H^*_T(f^{-1}(\kappa_0))$ is surjective.
\end{thm}

\begin{proof}
For a real number $a$, let
\[
M^a=f^{-1}((-\infty,a]).
\]
Let $\kappa$ be a critical value of $f$, and $B^\kappa_1,\ldots, B^\kappa_{j_\kappa}$ be the connected components of the critical set at level $\kappa$. Denote by $D^- B^\kappa_i$ and $S^- B^\kappa_i$ the disk respectively sphere bundle in the negative normal bundle (see e.g.~\cite[Section 1]{AtiyahBottYang}) of $B^\kappa_i$. We write $\lambda^\kappa_i$ for the rank of $D^-B^\kappa_i$. Let $E^\kappa_i\in H^*_T(B^\kappa_i)$ be the equivariant Euler class of $D^-B^\kappa_i$. Consider the following diagram, in which the top row is the long exact sequence of the pair $(M^{\kappa+\varepsilon},M^{\kappa-\varepsilon})$.
\[
\xymatrix{ \ldots \ar[r] & H^*_T(M^{\kappa+\varepsilon},M^{\kappa-\varepsilon}) \ar[r] \ar[d] & H^*_T(M^{\kappa+\varepsilon}) \ar[r] \ar[d] & H^*_T(M^{\kappa-\varepsilon}) \ar[r] & \ldots \\
& \bigoplus_i H^*_T(D^-B^\kappa_i,S^-B^\kappa_i) \ar[r] \ar[d] & \bigoplus_i H^*_T(D^-B^\kappa_i) \ar[d] \\
& \bigoplus_i H^{*-\lambda^\kappa_i}_T(B^\kappa_i) \ar[r]^{\bigoplus (\cdot E^\kappa_i)} & \bigoplus_i H^{*}_T(B^\kappa_i)  }
\]
The two vertical arrows on the left are isomorphisms (one because of excision, and the other is the inverse of the Thom isomorphism), and because $D^-B^\kappa_i$ has no $T^\kappa_i$-fixed vectors, where $T^\kappa_i$ is the isotropy group of $B^\kappa_i$, multiplication with the Euler classes is injective by \cite[Proposition 4]{Duflot}. Thus, $H^*_T(M^{\kappa+\varepsilon},M^{\kappa+\varepsilon})\to H^*_T(M^{\kappa+\varepsilon})$ is injective and we obtain short exact sequences
\[
0 \to H^*_T(M^{\kappa+\varepsilon},M^{\kappa-\varepsilon})\to H^*_T(M^{\kappa+\varepsilon})\to H^*_T(M^{\kappa-\varepsilon})\to 0.
\]
For the absolute minimum $\kappa_0$ of $f$, $H^*_T(M^{\kappa_0})$ is the sum of Cohen-Macaulay modules of Krull dimension $r-b$, and hence Cohen-Macaulay by Lemma \ref{lem:CMseq}. If for some critical value $\kappa$ we already know that $H^*_T(M^{\kappa-\varepsilon})$ is Cohen-Macaulay of Krull dimension $r-b$, the short sequence above, combined with Lemma \ref{lem:CMseq} and the fact that $H^*_T(M^{\kappa+\varepsilon},M^{\kappa+\varepsilon})$ is Cohen-Macaulay of Krull dimension $r-b$ as well, implies that $H^*_T(M^{\kappa+\varepsilon})$ is Cohen-Macaulay of the same Krull dimension.
\end{proof}

\begin{rem} Note that $f$ is automatically an equivariantly perfect Morse-Bott function as all components of $M_b$ are equivariantly self-completing, cf.~\cite[Prop.~1.9.]{AtiyahBott}
\end{rem}

\section{Equivariant disk bundle decompositions}\label{sec:diskbundle}

In this section, we will prove a generalization of a theorem of Harada, Henriques and Holm, see \cite{Harada}, in particular Theorem 2.2 of the unpublished version.  For us, a cell bundle is a $T$-equivariant oriented disk bundle over a compact $T$-manifold $Y$. The dimension of a cell bundle is the fiber dimension.

We say that a compact manifold $M$ has a $T$-invariant disk bundle decomposition if it can be built from a union of zero-dimensional cell bundles, and successively attaching cell bundles. The attaching maps are not required to map the boundary into smaller dimensional cell bundles.

\begin{thm}\label{thm:diskbundle} Let $k$ be a fixed integer. Let $M$ be a compact $T$-manifold that admits a finite $T$-invariant disk bundle decomposition into finitely many even-dimensional cell bundles $E\to Y$ that satisfy the following conditions:
\begin{enumerate}
\item $T$ acts on $Y$ with only one local isotropy type of dimension $k$ 
\item $\Hb^{odd}(Y)=H^{odd}(Y/T)=0$.
\end{enumerate} Then the action is Cohen-Macaulay. 
\end{thm} 
\begin{proof}
As the maximal dimension of an isotropy algebra is $k$, we know \cite[Proposition 5.1]{FranzPuppe2003} that the Krull dimension of $H^*_T(M)$ is $k$.

As in the definition of a disk bundle decomposition, let $M^0$ be the union of zero-dimensional cell bundles, and $M^i$ be the space obtained by attaching the first $i$ cell bundles. We prove by induction that  $H^*_T(M^i)$ is Cohen-Macaulay of Krull dimension $k$, and $H^{odd}_T(M^i)=0$. First, $H^*_T(M^0)$ is the sum of Cohen-Macaulay modules of the form $H^*_T(Y)=H^*(Y/T)\otimes S(\mft_Y^*)$, where $\mft_Y$ is the unique $k$-dimensional isotropy algebra of $Y$. So $H^*_T(M^0)$ is Cohen-Macaulay of Krull dimension $k$ by Lemma \ref{lem:CMseq}, and $H^{odd}_T(M^0)=0$.

We claim that in the long exact sequence of the pair $(M^{i+1},M^i)$, the boundary operator vanishes. For this, it is sufficient to show that $H^{odd}_T(M^{i+1},M^i)=0$. Let $E\to Y$ be the cell bundle that gets attached to $M^i$. Denote its dimension by $2d$ and the unique isotropy algebra of $Y$ by $\mft_Y$. Then
\begin{equation} \label{eq:attachingCM}
H^*_T(M^{i+1},M^i) = H^*_{T}(E,\partial E) \overset{\text{Thom}}{=} H^{*-2d}_T(Y)=\Hb^{*-2d}(Y)\otimes S(\mft_Y^*),
\end{equation}
so by assumption $(2)$, $H^{odd}_T(M^{i+1},M^i)$ vanishes. Thus, the boundary operator in the long exact sequence of the pair $(M^{i+1},M^i)$ vanishes and we obtain a short exact sequence
\[
0\to H^*_T(M^{i+1},M^i)\to H^*_T(M^{i+1})\to H^*_T(M^i)\to 0.
\]
It follows that $H^{odd}_T(M^{i+1})=0$. Because \eqref{eq:attachingCM} implies that $H^*_T(M^{i+1},M^i)$ is Cohen-Macaulay of Krull dimension $k$,  $H^*_T(M^{i+1})$ is also Cohen-Macaulay of Krull dimension $k$ by Lemma \ref{lem:CMseq}.
\end{proof}

\section{Equivariant cohomology and the bottom stratum}\label{sec:bottomstratum}
Although it is not an equivalent characterization of equivariant formality, the injectivity of the restriction map $H^*_T(M)\to H^*_T(M^\mft)$ due to the torsion-freeness of $H^*_T(M)$ is an important property of equivariantly formal actions. In this section, we replace $M^\mft$ by the bottom stratum $B$ of the action, and find an algebraic characterization of injectivity of $H^*_T(M)\to H^*_T(B)$ which has an interesting geometric consequence, see Theorem \ref{thm:OFA+injective}. This will be applied in Section \ref{sec:ActionsWithAcyclicOrbitSpace}.

If $T$ acts on a manifold $N$, then we say that $\alpha\in H^*_T(N)$ is {\it invisible} if no complexified nonregular isotropy algebra of the $T$-action on $N$ is contained in $\supp \alpha$. If $N$ is compact,  this is by Lemma \ref{lem:kernelM^K} equivalent to saying that 
for every nonregular isotropy algebra $\mfk$ of the action, $\alpha$ is in the kernel of $H^*_T(N)\to H^*_T(N^\mfk)$. This motivates the terminology.
 
\begin{prop} \label{prop:bottominjective} Let $T$ act on a compact manifold $M$, and let $B$ denote the bottom stratum of the action. Then the natural map
\[
H^*_T(M)\to H^*_T(B)
\]
is injective if and only if for every $p\notin B$, $H^*_{T}(M^p)$ does not contain any invisible element.
\end{prop}
\begin{proof} 
Assume first that $\alpha\in H^*_T(M^p)$ is invisible for $p\notin B$, i.e., the only complexified isotropy algebra contained in $\supp \alpha$ is $\mft_p\otimes \CC$. Because the equivariant push-forward map $H^*_T(M^p)\to H^*_T(M)$ is injective \cite[Proposition 4]{Duflot}, the image of $\alpha$ has the same support as $\alpha$ and thus maps to zero in $H^*_T(B)$ by Lemma \ref{lem:kernelM^K}.

Assume now that for $p\notin B$, $H^*_T(M^p)$ contains no invisible elements. We prove that $H^*_T(M)\to H^*_T(B)$ is injective by induction over the length $s$ of the longest chain $\mft_1 \subsetneq \mft_2\subsetneq \ldots \subsetneq \mft_s$
of isotropy algebras.

If this length is $s=1$, i.e., the action has only one local isotropy type, we have $M=B$ and the claim is trivial. Assume the proposition is proven for all actions with $s<m$, and let $T$ act on $M$ with $s=m>1$. 

Let $\mfk_j\neq \{0\}$ denote the minimal nonregular isotropy algebras of the action, i.e., if $\mfh\subset \mfk_j$ is any isotropy group, then either $\mfh$ is the regular isotropy group (i.e., $\{0\}$ if the action is effective) or $\mfh=\mfk_j$. Consider the map
\[
\eta: H^*_T(M)\to \bigoplus_j H^*_T(M^{\mfk_j}).
\]
By Lemma \ref{lem:kernelM^K}, the kernel of $\eta$ consists of those $\alpha$ whose support does not contain any $\mfk_j\otimes \CC$. But the $\mfk_j$ are exactly the minimal isotropy algebras, so we get
\[
\ker \eta=\{\alpha\in H^*_T(M)\mid  \alpha \text{ invisible}\}.
\]
By assumption (choose a regular $p$), such invisible elements do not exist, so $\eta$ is injective. The longest chains of isotropy algebras for the $T$-actions on $M^{\mfk_j}$ are shorter than the one of the action on $M$, so by induction we know that the maps $H^*_{T}(M^{\mfk_j})\to H^*_{T}(B\cap M^{\mfk_j})$, induced by the inclusions, are injective. It follows that 
\[
H^*_T(M)\to \bigoplus_j H^*_T(M^{\mfk_j})\to \bigoplus_j H^*_T(B\cap M^{\mfk_j})
\]
is injective. But this map consists only of copies of the natural maps induced by the inclusions of components of the bottom stratum, and thus, $H^*_T(M)\to H^*_T(B)$ has to be injective itself.
\end{proof}

Next, we find a geometric consequence of the existence of an invisible element $\alpha\in H^*_T(M)$. Duflot \cite[proof of Theorem 1]{Duflot} proves that for every $i$, the push-forward $(\varphi_i)_*:H^*_T(M_{(i)})\to H^*_T(M\setminus M_{i-1})$ of the inclusion $\varphi_i:M_{(i)}\to M\setminus M_{i-1}$ is injective, i.e., the corresponding Gysin sequence is in fact a short exact sequence
\begin{equation}\label{eq:GysinMi}
0\to H^*_T(M_{(i)})\to H^*_T(M\setminus M_{i-1})\to H^*_T(M\setminus M_i)\to 0.
\end{equation}  
Note that this is a sequence of $S(\mft^*)$-modules, see e.g.~the discussion on the equivariant push-forward map in \cite[\S 2]{AtiyahBott}, or \cite[p.~221]{GGK}. These considerations imply the following lemma:
\begin{lem}\label{lem:supp0} Let $T$ act on a compact manifold $M$. If $\alpha\in H^*_T(M)$ is invisible, then $\alpha$ is not in the kernel of the map $H^*_T(M)\to H^*_T(\Mr)$.
\end{lem}
\begin{proof} If we can show that an invisible element can never be in the kernel of $H^*_T(M\setminus M_{i-1})\to H^*_T(M\setminus M_i)$ for $i<r$, then the claim follows. This kernel is by \eqref{eq:GysinMi} the same as the image of $(\varphi_i)_*$. For any $\omega\in H^*_T(M_{(i)})$,  we have $\supp (\varphi_i)_*\omega=\supp \omega$ because of the injectivity of $(\varphi_i)_*$. But on the other hand, if we choose a point $p_j$ in each connected component of $M_{(i)}$, then $H^*_T(M_{(i)})=\bigoplus_j H^*_T(\Mr^{p_j})= \bigoplus_j H^*(\Mr^{p_j}/T)\otimes S(\mft_{p_j}^*)$ by Lemma \ref{lem:trivialpartofaction}, and the $\mft_{p_j}$ are nonregular isotropy algebras on $M\setminus M_{i-1}$.
\end{proof}

\begin{cor}\label{cor:nonexistencesupp0}
Let $T$ act nontransitively on a connected compact manifold $M$. If $\Mr/T$ is acyclic, i.e., $\Hb^*(\Mr)=H^*(\Mr/T)=\RR$, then there are no invisible elements in $H^*_T(M)$.
\end{cor}
\begin{proof}
Assume without loss of generality that the $T$-action is effective. Let $\alpha\in H^*_T(M)$ be invisible. By Lemma \ref{lem:supp0}, $\alpha$ defines a nonzero cohomology class in $H^*_T(\Mr)=H^*(\Mr/T)=\RR$, which implies that $\alpha$ is a $0$-form. In other words, $1\in H^*_T(M)$ is invisible. But the support of $1\in H^*_T(M)$ contains the complexification of every isotropy algebra, cf.~Proposition \ref{prop:characterizationlocallyfree}. This means that the action is locally free, i.e., $\Mr=M$. As $\Mr/T$ satisfies Poincar\'{e} duality, this is only possible if $\Mr/T=M/T$ is a point, i.e., the action is transitive.
\end{proof}

\begin{dfn}\label{dfn:ofaetc}
We call $M^p/T$ a {\it face} of the orbit space.
We say that $M/T$ is {\it closed-face-acyclic} if its closed faces are acyclic; that means that for any point $p\in M$, we have $\Hb^*(M^p)=\RR$. If the open faces of $M/T$ are acyclic,  i.e., if for all $p$ we have $\Hb^*(\Mr^p)=\RR$, we call $M/T$ {\it open-face-acyclic}. We say that the orbit space is  {\it almost open-face-acyclic} if we have $\Hb^*(\Mr^p)=\RR$ for all $p$ not contained in the bottom stratum $B$ of the action.
\end{dfn}

\begin{rem} For torus manifolds, the notion of face-acyclicity as introduced by \cite{MasudaPanov} is the same as our condition of having closed-face-acyclic orbit space, but note that they use integer coefficients instead of the reals.
\end{rem}

With this notation, we obtain
\begin{thm}\label{thm:OFA+injective} If $M/T$ is almost open-face-acyclic, then the natural map $H^*_T(M)\to H^*_T(B)$ is injective.
\end{thm}
\begin{proof} This is just Proposition \ref{prop:bottominjective}, combined with Corollary \ref{cor:nonexistencesupp0} for every $M^p$.
\end{proof}

\section{Actions with face-acyclic orbit space}\label{sec:ActionsWithAcyclicOrbitSpace}

The goal of this section is to prove that actions with open-face-acyclic orbit space (see Definition \ref{dfn:ofaetc}) on orientable compact manifolds are Cohen-Macaulay, see Theorem \ref{thm:OFAisCM}. In Subsection \ref{sec:OFA} we investigate the general topological structure of actions with (almost) open-face-acyclic orbit space, such as the dimensions of the strata or the structure of the bottom stratum. In particular,  if $b$ is the dimension of the smallest occuring orbit, $M_{b+1}/T$ will be seen to be a connected graph, just as the $1$-skeleton of an equivariantly formal action satisfying the GKM conditions, see Subsection \ref{sec:b+1}.  In Subsections \ref{sec:Chang} and \ref{sec:OFAisCM} we finish the proof that actions with open-face-acyclic orbit space are Cohen-Macaulay, following ideas by Masuda and Panov \cite{MasudaPanov}, but with certain differences as mentioned in the introduction. Combined with a statement of Bredon \cite{Bredon}, our result characterizes equivariantly formal actions on orientable compact manifolds $M$ with $\dim M=2\dim T$ via the orbit space, see Subsection \ref{sec:equivformalcase}.

\subsection{Open- and closed-face-acyclicity}\label{sec:OFA}

For the following proposition, cf.~also the remark after Proposition 9.3 of \cite{MasudaPanov}.
\begin{prop}\label{prop:OFA+}
If the orbit space of an effective $T$-action on an oriented compact manifold $M$ is almost open-face-acyclic, then one of the following statements holds:
\begin{enumerate}
\item $B$ is connected.
\item The orbit space is open-face-acyclic, i.e., $B$ consists of finitely many isolated orbits. Moreover, they are all of the same dimension, i.e., there exists $b$ such that $M_b=B$. In addition, $M_{b+1}$ is connected.
\end{enumerate} 
\end{prop}
\begin{proof} Assume that $B$ is not connected. In particular, the action is not locally free and hence has at least two local isotropy types. 

Let $B_0$ be a component of $B$. We have to show that $B_0$ is an orbit. The components $M^p$ of isotropy manifolds are partially ordered by inclusion, the unique maximal element being $M$ itself. Since the bottom stratum of $M$ is disconnected by assumption, we may choose a minimal such component $M^p$ with the property that the bottom stratum of the $T$-action on $M^p$ is disconnected, but contains $B_0$. Let the other components of the bottom stratum of $M^p$ be denoted by $B_1,\ldots,B_s$.

We claim that $\Ms^p$ is disconnected. Assume this is not the case, and recall
\begin{equation*}
\Ms^p=\bigcup_{q\in M^p:\, \mft_p\neq \mft_q} M^q.
\end{equation*}
For every $q\in M^p$ with $\mft_p\neq \mft_q$, the bottom stratum of the $T$-action on $M^q$ is connected, and hence one of the $B_i$, $i=0,\ldots,s$. Further, for two such singular points $q_1,q_2\in M^p$ such that $M^{q_1}\cap M^{q_2}\neq \emptyset$, their respective bottom strata are, because of connectedness, contained in $M^{q_1}\cap M^{q_2}$ and coincide. Therefore, if we let 
\[
N_i:=\bigcup_{q:\, \text{the bottom stratum of }M^q \text{ is }B_i} M^q,
\]
then $\Ms^p=\bigcup_i N_i$ is a disjoint union into nonempty subsets, and hence $\Ms^p$ is disconnected. 

Consider now the first terms of the long exact sequence in basic cohomology of the pair $(M^p,\Ms^p)$:
\[
0\to \Hb^0(M^p)\to \Hb^0(\Ms^p)\to \Hb^1(M^p,\Ms^p)\to \ldots
\]
It follows that $\Hb^1(M^p,\Ms^p)\neq 0$. But since $\Hb^1(M^p,\Ms^p)=H^1_{c}(\Mr^p/T)$, and $\Mr^p/T$ satisfies Poincar\'e duality, the almost open-face-acyclic condition implies that the $T$-action on $M^p$ has cohomogeneity one. In particular, it has exactly two singular orbits, one of which is $B_0$. Because $T$ is a torus, the regular orbits are $S^1$-fibre bundles over the singular orbits, and hence the two singular orbits have equal dimension.

It follows that $B$ consists of isolated orbits, and whenever two of those orbits can be joined by a sequence of $M^p$'s such that the $T$-action on $M^p$ has cohomogeneity one (we say: they are \emph{linked}), they are of the same dimension. We still need to show that any two of those orbits are linked. If this was not the case, choose an arbitrary component $B_0$, and let $M^p$ be minimal with the property that the bottom stratum of $M^p$ has a component linked with $B_0$ and a component not linked with $B_0$. In other words, for all singular $q\in M^p$, the components of the bottom stratum of $M^q$ are either all linked with $B_0$ or all not linked with $B_0$. By an analogous argument as above, $\Ms^p$ is disconnected. But this implies that the $T$-action on $M^p$ has cohomogeneity one, which is a contradiction.
\end{proof}
\begin{ex}
An easy example for an action whose orbit space is almost open-face-acyclic but not open-face-acyclic, is the $S^1$-action on $M=S^3=\{(z,w)\mid |z|^2+|w|^2=1\}\subset \CC^2$ is given by $t\cdot (z,w)=(tz,w)$. The orbit space $M/S^1$ is the disk $D^2$, and $M^{S^1}=\partial D^2=S^1$.
\end{ex}

For the rest of the section, $b$ will denote the dimension of the smallest occuring orbit.

\begin{lem} \label{lem:extensionforms} Let $T$ act locally freely on a manifold $M$ (e.g.~the regular stratum of another $T$-action) such that $\Hb^2(M)=0$. Then for every $p\in M$, the map in cohomology $H^*(M)\to H^*(T)$ induced by the orbit map $T\to M; \, t\mapsto tp$, is surjective.
\end{lem}
\begin{proof} It suffices to prove that $H^1(M)\to H^1(T)$ is surjective, as $H^*(T)$ is generated by $H^1(T)$. Fix a basis $X_1,\ldots,X_r$ of $\mft$, with dual basis $u_1,\ldots,u_r\in \mft^*$.  Let $\omega:TM\to \mft$ be a connection form, i.e., a $T$-invariant map such that for every $p\in M$, we have $\omega(X_i(p))=X_i$. Note that a choice of connection form is equivalent to the choice of a $T$-invariant horizontal distribution, e.g., the orthogonal complement of the orbit with respect to some $T$-invariant Riemannian metric on $M$. 

We obtain one-forms $\theta_i=u_i\circ \omega$ on $M$ (these are sometimes called connection forms, see e.g.~\cite[p.~23]{GuilleminSternberg}). We claim that $d\theta_i$ is a closed basic two-form. 
The $d\theta_i$ are clearly $T$-invariant, as $\omega$ is $T$-invariant. Thus,
\[
i_{X_j} d\theta_i = di_{X_j} (u_i\circ \omega)=d (u_i(X_j))=d\delta_{ij}=0,
\]
and hence $d\theta_i$ is a closed basic two-form. By our assumption it follows that the $d\theta_i$ are basic exact, i.e., there exist basic one-forms $\eta_i$ such that the $\theta_i-\eta_i$ are closed and thus define cohomology classes on $M$. The pull-back of $\theta_i-\eta_i$ via an orbit map $\varphi_p$ of a point $p$ is the left-invariant one-form given by $u_i$ since
\[
\varphi^*_p(\theta_i-\eta_i)(X_j)=u_i ( \omega (d\varphi_p (X_j)))= u_i(X_j)=\delta_{ij}.
\]
Since those left-invariant one-forms span $H^1(T)$, the claim follows.
\end{proof}
For $p\in M$, let $d_p=\dim M^p-\dim T\cdot p=\dim M^p/T$. Observe that for $q\in M^p$, we have $d_q\leq d_p$, and if $q\in \Ms^p$, even $d_q<d_p$.
\begin{prop}\label{prop:OFAd_p} If the orbit space of a $T$-action on an orientable manifold $M$ is open-face-acyclic, then for all $p\in M_i\setminus M_{i-1}$ we have $d_p=i-b$. 
\end{prop}
\begin{proof} The subsequent lemma implies that for every $q$ and $p$ such that $M^q\subset M^p$ and such that there is no $q'\in M^p$ with $M^q\subsetneq M^{q'}\subsetneq M^p$, we have $d_p=d_q+1$ (apply the lemma to the $T/T_p$-action on $M^p$). The claim follows by induction, since for $p\in M_b=B$ it is clear by Proposition \ref{prop:OFA+}.
\end{proof}
\begin{lem}\label{lem:OFAd_p} If the $T$-action is effective and $\Hb^2(\Mr)=0$, then for every $p\in \Ms$ such that there is no $q\in M$ with $M^p\subsetneq M^q\subsetneq M$, we have $d_p=\dim M/T-1$.
\end{lem}
\begin{proof} By Lemma \ref{lem:extensionforms}, the condition on the basic cohomology implies that for any regular point $q$, the orbit map of $q$ induces a surjective map $H^*(\Mr)\to H^*(T)$. 
If $p$ is such as in the statement of the lemma, then we can choose $\varepsilon$ so small that $N=\exp( S^\varepsilon(\nu M^p)|_{T\cdot p})$, where $S^\varepsilon(\nu M^p)$ is the sphere bundle of radius $\varepsilon$ in the normal bundle of $M^p$, is contained in the regular stratum. We furthermore may assume that $\exp$, restricted to $S^\varepsilon(\nu M^p)|_{T\cdot p}$,  is a diffeomorphism onto its image. Then, the map $H^*(\Mr)\to H^*(T)$, induced by the orbit map of some point in $N$, factors through $H^*(N)$, and we obtain a surjective map $H^*(N)\to H^*(T)$. On the other hand, 
\[
S^\varepsilon(\nu M^p)|_{T\cdot p}=T\times_{T_p} S^\varepsilon(\nu_p M^p) = (T\times_{T_p^0} S^\varepsilon(\nu_p M^p))/(T_p/T_p^0).
\]
As $T_p/T_p^0$ is a finite group and we are dealing with $\RR$-coefficients, we have $H^1(N)=H^1(T\times_{T_p^0} S^\varepsilon(\nu_pM^p))^{T_p/T_p^0}$, see e.g.~Borel et al.~\cite[Cor.~III.2.3]{Borel}. If $T'$ is some complement of the identity component $T_p^0$ in $T$, i.e., $T=T_p^0\times T'$, we have $T\times_{T_p^0} S^\varepsilon(\nu_pM^p)=T'\times S^\varepsilon(\nu_pM^p)$. Since $T'$ is strictly lower dimensional than $T$, it follows that $H^1(N)$ can only map surjectively onto the $r$-dimensional space $H^1(T)$ if the sphere $S^\varepsilon(\nu_pM^p)$ and $T_p$ are one-dimensional. But this implies that $M^p$ has codimension two in $M$, and thus $d_p=\dim M-2-\dim T\cdot p = \dim M-\dim T -1 = \dim M/T -1$, which finishes the proof.
\end{proof}

\begin{cor}\label{cor:OFArelcohom} If the orbit space of a $T$-action on a compact orientable manifold $M$ is open-face-acyclic, then 
\begin{enumerate}
\item $\Hb^k(M_i,M_{i-1})=0$ for $k\neq i-b$. The dimension of $\Hb^{i-b}(M_i,M_{i-1})$ is equal to the number of components of $M_i\setminus M_{i-1}$.
\item For $j<i$, the natural map $\Hb^{j-b}(M)\to \Hb^{j-b}(M_i)$ is an isomorphism.
\item For $j>i$, $\Hb^{j-b}(M_i)=0$.
\end{enumerate}
\end{cor}
\begin{proof}
The first part follows from Proposition \ref{prop:OFAd_p} because every $\Mr^p/T$ satisfies Poincar\'{e} duality. 
For the second part, write $\Hb^{j-b}(M)\to \Hb^{j-b}(M_i)$ as the composition
\[
\Hb^{j-b}(M)\to \Hb^{j-b}(M_{r-1})\to \ldots\to  \Hb^{j-b}(M_{i+1})\to \Hb^{j-b}(M_i).
\]
Using the exact cohomology sequences of the respective pairs, the first part of the corollary implies that each of those maps is an isomorphism if $j<i$. 

The same argument gives that each map in the composition
\[
\Hb^{j-b}(M_i)\to \Hb^{j-b}(M_{i-1})\to\ldots \to \Hb^{j-b}(M_b) \to \Hb^{j-b}(M_{b-1})=0
\]
is an isomorphism if $j>i$, hence the third part follows.
\end{proof}

\begin{cor} \label{cor:dimMdimTb} If the orbit space of an effective $T$-action on a compact orientable manifold $M$ is open-face-acyclic, then $\dim M= 2\dim T - b$.
\end{cor}
\begin{proof} Proposition \ref{prop:OFAd_p} implies $d_p=\dim T-b$ for regular $p$. Thus,  $\dim M=\dim M/T+\dim T=2\dim T-b$.
\end{proof}

\begin{cor} \label{cor:OFAweights} If the orbit space of an effective $T$-action on an orientable compact manifold $M$ is open-face-acyclic, then for all $p\in M$, we have $\dim \nu_p M^p=2\dim T_p$. Consequently, the natural $T_p$-representation on the normal space $\nu_p M^p$ has exactly $\dim T_p$ weights.
\end{cor}
\begin{proof} For $p\in M_i$, we have $\dim T_p=\dim T-i$ and thus
$\dim \nu_pM^p = 2\dim T-b - \dim M^p = 2\dim T -b -(d_p+i)=2\dim T_p.$ It follows that the $T_p$-representation has at most $\dim T_p$ many weights. Because of effectiveness, the second claim follows.
\end{proof}

For any $q\in M^p$, the isotropy representation at $q$ induces $T_q$-representations on the tangent and normal spaces $T_qM^p$ and $\nu_qM^p$. We have $\nu_qM^q=\nu_qM^p\oplus (\nu_qM^q\cap T_qM^p)$ as $T_q$-modules. Corollary \ref{cor:OFAweights} implies that $\nu_qM^p=\bigoplus_{i=1}^{\dim \mft_p}V_{\beta_i}$ and $\nu_qM^q\cap T_qM^p=\bigoplus_{i=1}^{\dim \mft_q-\dim \mft_p} V_{\gamma_i}$, where the $V_{\beta_i}$ and $V_{\gamma_i}$ are the weight spaces of the respective weights $\beta_i,\gamma_i\in \mft_q^*$.
Because $\mft_p\subset \mft_q$ acts trivially on $\nu_qM^q\cap T_qM^p$, we have $\mft_p=\bigcap_{i=1}^{\dim \mft_q-\dim \mft_p}\ker \gamma_i$. The restriction map $\beta_i\mapsto \left.\beta_i\right|_{\mft_p}$ is a one-to-one-correspondence between the weights of the $T_q$- and the $T_p$-representations on $\nu_qM^p$.  The weights are constant along $M^p$ in the following sense:
\begin{lem}\label{lem:normalbundlesplits}
The weights of the $T_p$-representation on $\nu_qM^p$, coincide with the weights $\alpha_i$ of the $T_p$-representation on $\nu_pM^p$. Moreover, the normal bundle $\nu M^p$ splits equivariantly as $\nu M^p=\bigoplus_{i=1}^k V_{\alpha_i}$.
\end{lem}

\begin{proof} This is essentially Proposition 1 of \cite{Duflot}, which is an extended version of Proposition 1.6.2 of \cite{Atiyah} for tori instead of finite groups. 
\end{proof}

\begin{prop} \label{prop:OFACFA} If the orbit space of a $T$-action on an orientable compact manifold $M$ is open-face-acyclic, then it is also closed-face-acyclic, i.e., $\Hb^*(M^p)=\RR$ for all $p\in M$. In particular, $\Hb^*(M)=\RR$.
\end{prop}
\begin{proof} We show that the map $\Mr\to M$ induces an isomorphism on basic cohomology by regarding it as the composition $\Mr=M\setminus M_{r-1}\to \ldots \to M\setminus M_b \to M$. To do this, we have to show that for all $i$, the map $\Hb^*(M\setminus M_{i-1})Ê\to \Hb^*(M\setminus M_{i})$ is an isomorphism. Choose a point $p$ in each connected component of $M_i\setminus M_{i-1}$, together with disjoint neighborhoods $U_p$ of $\Mr^p$ that have $\Mr^p$ as strong equivariant deformation retracts, and are diffeomorphic to the normal bundles $\nu \Mr^p$.
\[\tiny{
\xymatrix@R=20pt@C=10pt{
\Hb^*(M\setminus M_{i-1},M\setminus M_{i}) \ar[d] \ar[r] & \Hb^*(M\setminus M_{i-1}) \ar[d] \ar[r] & \Hb^*(M\setminus M_i) \ar[d] \ar[r] & \Hb^{*+1}(M\setminus M_{i-1},M\setminus M_{i}) \ar[d] \\
\bigoplus \Hb^*(U_p,U_p\setminus \Mr^p)\ar[r]  & \bigoplus \Hb^*(U_p) \ar[r] & \bigoplus \Hb^*(U_p\setminus \Mr^p) \ar[r] & \bigoplus \Hb^{*+1}(U_p,U_p\setminus \Mr^p)   } }
\]
In the diagram, the vertical maps on the sides are isomorphisms by excision. Thus, we need to show that the maps $\Hb^*(U_p)\to \Hb^*(U_p\setminus \Mr^p)$ are isomorphisms, which by our choice of $U_p$ amounts to show that for every $p\in M$, the spaces $(\nu \Mr^p)/T$ and $(\nu \Mr^p \setminus \Mr^p)/T$ are homotopy equivalent.

Let $\alpha_i$, $i=1\ldots k$ be the weights of the normal bundle $\nu \Mr^p=\bigoplus V_{\alpha_i}$ as in Lemma \ref{lem:normalbundlesplits}. Any vector $v\in \nu_q\Mr^p$ can be uniquely written  as $v=\sum v_i$, where $v_i\in V_{\alpha_i}(q)$. Because the $V_{\alpha_i}$ are $T$-invariant, this defines a $T$-invariant map
\[
r:\nu \Mr^p\to \RR^k_{\geq 0};\quad v\mapsto (||v_1||,\ldots,||v_k||).
\]
It follows that the map
\[
(\nu \Mr^p)/T \to \Mr^p/T \times \RR^k_{\geq 0};\quad T\cdot v \mapsto (T\cdot q,r(v))
\]
is well-defined. It is clearly surjective, and injectivity follows because the $T_q$-orbits in $\nu_q\Mr^p$ are products of $S^1$-orbits in $V_{\alpha_i}(q)$ by Corollary \ref{cor:OFAweights}.
Noting that under this map, $\Mr^p/T$ corresponds to $\Mr^p/T\times \{0\}$, the desired homotopy equivalence follows.
\end{proof}

\begin{prop}\label{prop:basicABsequence} If the orbit space of a $T$-action on an orientable compact manifold $M$ is open-face-acyclic, then the sequence 
$$0\to \Hb^*(M)\to \Hb^*(M_b)\overset{\partial_b}{\to} \Hb^*(M_{b+1},M_b)\overset{\partial_{b+1}}{\to} \ldots \overset{\partial_{r-1}}{\to} \Hb^* (M_r,M_{r-1})\to 0,$$
where the first map is induced by the inclusion, is exact.
\end{prop}
\begin{proof} 
The proof is the same as the standard proof that the cellular cohomology of a CW complex computes the standard cohomology.  Let $\omega\in \Hb^{i-b}(M_i,M_{i-1})$ cause nonexactness of the sequence
\[
\Hb^*(M_b)\overset{\partial_b}{\to} \Hb^*(M_{b+1},M_b)\overset{\partial_{b+1}}{\to} \ldots \overset{\partial_{r-1}}{\to} \Hb^* (M_r,M_{r-1})\to 0
\]
at $\Hb^{i-b}(M_i,M_{i-1})$ for some $i>b$, i.e., $\partial_i \omega=0$ but $\omega\notin \im \partial_{i-1}$. Because $\partial_{i}\omega=0$, the exact sequence of the triple $(M_{i+1},M_i,M_{i-1})$ in basic cohomology implies that  there is an element $\eta\in \Hb^{i-b}(M_{i+1},M_{i-1})$ that is mapped to $\omega$ under the natural restriction map. We claim that $\eta$ defines a nontrivial element in $\Hb^{i-b}(M_{i+1})$. If that was not the case, $\eta$ would be in the image of the boundary map $\Hb^{i-b-1}(M_{i-1})\to \Hb^{i-b}(M_{i+1},M_{i-1})$. But since $\Hb^{i-b-1}(M_{i-2})=0$ by the third part of Corollary \ref{cor:OFArelcohom}, this would produce a contradiction to the assumption that $\omega\notin \im \partial_{i-1}$. Thus, $\Hb^{i-b}(M_{i+1})$ is nontrivial. By the second part of Corollary \ref{cor:OFArelcohom}, this implies that $\Hb^{i-b}(M)\neq 0$, which is in contradiction to Proposition \ref{prop:OFACFA}.  

To finish the proof, either use Proposition \ref{prop:vorneexakt} for basic cohomology, or redo the same argument as above to show that $\ker \partial_b$ is one-dimensional. 
\end{proof}

\begin{rem}\label{rem:basicsequencenotnecessary} This is a version of the Atiyah-Bredon sequence for basic cohomology. Unfortunately, we do not have a proof of the Cohen-Macaulayness of torus actions whose orbit space is open-face-acyclic that makes use of this sequence.

Note that for a general Cohen-Macaulay action, the basic version of the Atiyah-Bredon sequence is not necessarily exact. For example, consider the $S^1$-action on the $4$-sphere $M=S^4=\{(z,w,s)\mid |z|^2+|w|^2+s^2=1\}\subset \CC^2\times \RR$ given by
\[
t\cdot (z,w,s)=(tz,tw,s).
\]
It has exactly two fixed points and is thus equivariantly formal by criterion $(2)$ listed in Section \ref{sec:EquivariantFormality}. But the boundary operator $\partial_0:\Hb^*(M^{S^1})\to \Hb^*(M,M^{S^1})$  is not surjective, as $\Hb^3(M,M^{S^1})=H^3_c(\Mr/S^1)\cong H^0(\Mr/S^1)=\RR \neq 0$.
\end{rem}

\subsection{The $b+1$-skeleton}\label{sec:b+1}

By the results of Section \ref{sec:OFA} (in particular Propositions \ref{prop:OFA+}, \ref{prop:OFAd_p}, and Corollary \ref{cor:OFAweights}), the $b+1$-skeleton $M_{b+1}$ of an action on an orientable compact manifold $M$ with open-face-acyclic orbit space behaves similarly to the one-skeleton of an equivariantly formal action satisfying the so-called GKM conditions (see e.g.~\cite[Section 11.8]{GuilleminSternberg}). The only difference is that instead of being composed of two-spheres, $M_{b+1}$ is a union of submanifolds on which $T$ acts with cohomogeneity one. If two such cohomogeneity one submanifolds meet, their intersection consists of either one or two $b$-dimensional orbits. $M_{b+1}/T$ can be thought of as a graph, with the elements of $M_b/T=B/T$ as vertices. The vertices will be identified with the corresponding $b$-dimensional orbits, and also with points on that orbit. For $p\in M_b$ we write $[p]$ for $Tp$, when understood as a vertex. There is one edge for every submanifold $M^p$ with $p\in M_{b+1}\setminus M_{b}$ connecting its two $b$-dimensional orbits. Multiple edges between two vertices can occur. This graph is $d$-valent, with $d=\dim T-b$. Although the orbit space $M/T$ is not necessarily a convex polytope, we will refer to the $M^p/T$ as faces.
\begin{rem} 
In view of Remark \ref{rem:reductiontoequivformal}, if we choose a $b$-dimensional subtorus $K\subset T$ acting locally freely, then the $T/K$-action on $M/K$ satisfies the usual GKM conditions.
\end{rem}
\begin{ex} Consider the $T^3$-action on $M=S^5=\{(z_i)\in \CC^3\mid \sum |z_i|^2=1\}$ given by $(t_1,t_2,t_3)\cdot (z_1,z_2,z_3)=(t_1z_1,t_2z_2,t_3z_3)$. The bottom stratum $M_1$ consists of the three one-dimensional orbits, which are circles, and $M_2/T^3$ is a triangle. 
The diagonal $S^1\subset T^3$ acts freely on $S^5$, with $S^5/S^1=\CC P^2$. The induced $T^2$-action on $\CC P^2$ has the same orbit space as the $T^3$-action on $S^5$.
\end{ex}

For every oriented edge $e=M^p/T$, let $i(e)$ denote the initial vertex, $t(e)$ the terminal vertex, and we write $M_e=M^p$. There is a unique weight $\alpha(e)\in \mft_{i(e)}^*$ of the $T_{i(e)}$-representation on $\nu_{i(e)} T i(e)$ with kernel $\mft_p$.
\begin{lem} \label{lem:edgesschieben} For every two edges $e$ and $f$ with $i(e)=i(f)$, there is a unique edge $\tilde{e}$ with $i(\tilde{e})=t(f)$ such that $\left.\alpha(e)\right|_{\mft_f}=\left.\alpha(\tilde{e})\right|_{\mft_f}$.
\end{lem}
\begin{proof}
This is a consequence of Lemma \ref{lem:normalbundlesplits} and the discussion preceding it.
\end{proof}

\subsection{A Chang-Skjelbred Lemma}\label{sec:Chang}
Although the statements here are more general, most of the arguments in this section are taken from Sections 6 and 7 of \cite{MasudaPanov}. As before, we consider a $T$-action on an orientable compact manifold $M$ with open-face-acyclic orbit space, and the bottom stratum being the union of the $b$-dimensional orbits.

For a face $F=M^p/T$, let $\tau_F\in H^*_T(M)$ be the equivariant Thom class of $M^p$ in $M$ with respect to any orientation of the normal bundle, and $E_F\in H^*_T(M^p)$ the equivariant Euler class of the normal bundle of $M^p$. The restriction of $\tau_F$ to $M^p$ is $E_F$, see e.g.~\cite[p.~221]{GGK}.

\begin{lem} \label{lem:Eulerclassrestr}
Let $F=M^p/T$ be a face, and $2k$ the codimension of the closed $T$-invariant submanifold $M^p$ in $M$. Then for all $q\in M_b\cap M^p$,
$$
\left.E_F\right|_{T q}=\prod_{e\not\subset F,\ i(e)=[q]} \alpha(e)\in S^k(\mft_q^*)= H^{2k}_T(T q).
$$  
For $q\in M_b\setminus M^p$, we have $\left.E_F\right|_{T q}=0$.
\end{lem}
\begin{proof} For $q\notin M^p$, the statement is obvious, so let $q\in M_b\cap M^p$. By Lemma \ref{lem:normalbundlesplits}, the normal bundle $\nu M^p$ splits as the sum of $T$-equivariant two-plane bundles $\nu M^p=\bigoplus_{e\not\subset F,\ i(e)=[q]} V_{\alpha(e)}$. Thus, $E_F=\prod_{e\not\subset F,\ i(e)=[q]} E(V_{\alpha (e)})$. The bundle $V_{\alpha(e)}$, restricted to $Tq$, is
\[
\left.V_{\alpha(e)}\right|_{Tq}= T\times_{T_q} \CC = (T_q'\times \CC)/(T_q/T_q^0)
\]
where $T_q'$ is a complement of the identity component $T_q^0$ in $T$, i.e., $T=T_q\times T_q'$. 

We calculate the Euler class of the $T$-equivariant bundle $T_q'\times \CC\to T/T_q^0=T'$ in a way similar to \cite[Lemma I.3]{GGK}. See also \cite{BottTu} for the description of the equivariant Euler class in the Cartan model. Note that this bundle is trivial as a $T_q'$-equivariant bundle. For this, we choose a $T$-invariant connection form $\Theta\in \Omega^1(T_q'\times S^1)^T$ such that the $T_q'$-orbits are horizontal. Denoting the natural projection $T'\to Tq$ by $\rho$, one shows as in Lemma I.3 of the reference above that
\[
\rho^*(E_F)(\xi)=E(T_q'\times \CC)(\xi)=d_T\Theta(\xi)=\alpha(e)(\pi_{\mft_q}(\xi))
\]
for all $\xi\in \mft$, where $\pi_{\mft_q}:\mft\to \mft_q$ is the projection along the decomposition $\mft=\mft_q\oplus \mft_q'$. The form part vanishes because the curvature of $\Theta$ is zero by choice of $\Theta$. 
Using the isomorphism $\rho^*:S(\mft_q')\cong H^*_T(Tq) \to H^*_T(T/T_q^0)$, the claim follows.
\end{proof}
Note that because of the lemma, the restricted Euler class $\left.E_F\right|_{Tq}$ is independent of the chosen orientation of the normal bundle.

The next lemmas are analogous to Lemma 6.2 and 7.3 of \cite{MasudaPanov}; the only difference is that we are not allowed to subtract the restrictions of equivariant differential forms to different components of the bottom stratum.
\begin{lem}\label{lem:einschraenkungen} Let $N$ be a closed invariant subspace of $M$ containing $M_{b+1}$, and $\omega\in H^*_T(N)$. Choose an edge $e$ and write $[p]=i(e)$ and $[q]=t(e)$. Then the polynomials $\left.\omega\right|_{Tp}\in S(\mft_p^*)=H^*_T(Tp)$ and $\left.\omega \right|_{Tq}\in S(\mft_q^*)=H^*_T(Tq)$ coincide on the intersection $\mft_e=\mft_p\cap \mft_q$, where $\mft_e$ is the isotropy algebra of $M_e$.
\end{lem}
\begin{proof} Using Lemma \ref{lem:trivialpartofaction}, we obtain the following diagram, in which the upper right space is nothing but $S(\mft_p^*)\oplus S(\mft_q^*)$.
\[
\xymatrix@R=20pt@C=10pt{ H^*_T(M) \ar[r] \ar[d] & H^*_T(Tp)\oplus H^*_T(Tq) \ar[d] \ar[r]^<<<{\cong} & (H^*_{T/T_e}(Tp)\otimes S(\mft_e^*))\oplus (H^*_{T/T_e}(Tq)\otimes S(\mft_e^*)) \ar[d] \\
H^*_{T_e}(M_e) \ar[r] & H^*_{T_e}(Tp)\oplus H^*_{T_e}(Tq) \ar[r]^<<<<<{\cong} & (H^*(Tp)\otimes S(\mft_e^*))\oplus (H^*(Tq)\otimes S(\mft_e^*))
}
\]
Since the diagram commutes and $H^*_{T_e}(M_e)=H^*(M_e)\otimes S(\mft_e^*)$, we see that the image of $\omega$ in the bottom right space is of the form $(\sum_i \left.\omega_i\right|_{Tp}\otimes f_i, \sum_i \left.\omega_i\right|_{Tq} \otimes f_i)$ for some $f_i\in S(\mft_e^*)$ and $\omega_i\in H^*(M_e)$. As the restrictions of $\left.\omega\right|_{Tp}$ and $\left.\omega \right|_{Tq}$ to $S(\mft_e^*)$ are given by those summands for which $\omega_i$ is a $0$-form, the lemma follows.
\end{proof}

Let $F$ be a face of $M/T$, and $[q]\in F$ a vertex of $F$. We denote by $I(F)_{[q]}\subset S(\mft_q^*)$ the ideal generated by all $\alpha(e)$ with $e\subset F$. 

\begin{lem} \label{lem:CShelp2} Let $N$ be a closed invariant subspace of $M$ containing $M_{b+1}$, and let $F$ be a face of $M/T$. For every $\omega\in H^*_T(N)$, if $\left.\omega\right|_{Tp}\notin I(F)_{[p]}$ for some vertex $[p]\in F$, then $\left.\omega \right|_{Tq}\notin I(F)_{[q]}$ for every vertex $[q]\in F$.
\end{lem}
\begin{proof} Suppose $\left.\omega \right|_{Tq}\in I(F)_{[q]}$ for some vertex $[q]\in F$, i.e.
\[
\left.\omega \right|_{Tq}=\sum_{e\subset F:\, i(e)=[q]} \alpha(e) g_e 
\]
for some $g_e\in S(\mft_q^*)$. Now choose another vertex $[\tilde{q}]$ that is joined to $[q]$ by an edge $f$. By Lemma \ref{lem:edgesschieben}, for every $e$ in the sum above there is a unique edge $\tilde{e}\subset F$ with $i(\tilde{e})=[\tilde{q}]$ such that $\left.\alpha(\tilde{e})\right|_{\mft_f}=\left.\alpha(\tilde{e})\right|_{\mft_f}$.
If we define
\[
\eta:=\sum_e \alpha(\tilde{e})g_{\tilde{e}}\in I(F)_{[\tilde{q}]}
\]
where $g_{\tilde{e}}\in S(\mft_{\tilde{q}})$ is any polynomial with $\left.g_{\tilde{e}}\right|_{\mft_f}=\left.g_e\right|_{\mft_f}$, then Lemma \ref{lem:einschraenkungen} implies that $\left.\omega\right|_{T\tilde{q}}-\eta\in S(\mft_{\tilde{q}}^*)$ vanishes on $\mft_f$. In particular it is divisible by the weight that vanishes on $\mft_f$, and hence $\left.\omega\right|_{T\tilde{q}}\in I(F)_{[\tilde{q}]}$.

This completes the proof because the $b+1$-skeleton of $F$ is connected by Proposition \ref{prop:OFA+}.
\end{proof}

For every closed invariant subspace $N\subset M$ containing $M_{b+1}$, define 
\[
K_N=\ker(H^*_T(N)\to H^*_T(M_b)).
\]
 For $b=0$, i.e., the bottom stratum $M_b=M_0$ being the fixed point set, $K_N$ is the torsion submodule of $H^*_T(N)$.

\begin{prop} \label{prop:genbythomclasses} For every closed invariant subspace $N\subset M$ containing $M_{b+1}$, the quotient $H^*_T(N)/K_N$ is generated by the restrictions of the elements $\tau_F$ to $N$.
\end{prop}
\begin{proof} Now that we have transferred the necessary Lemmata \ref{lem:Eulerclassrestr} and \ref{lem:CShelp2} to our situation, the proof is exactly the same as the proof of Proposition 7.4 in \cite{MasudaPanov}, and we therefore omit it. It does not complicate things to consider $N$ instead of the whole manifold $M$, but note that it is important that $N$ contains the whole $(b+1)$-skeleton.
\end{proof}
As a corollary we obtain a version of the Chang-Skjelbred Lemma for actions with open-face-acyclic orbit space, compare \cite[Lemma 2.3]{ChangSkjelbred}. 
\begin{cor} \label{cor:torsionspaltet}
\begin{enumerate}
\item For every closed invariant subspace $N\subset M$ containing $M_{b+1}$, we have $H^*_T(N)=(\im H^*_T(M)\to H^*_T(N))\oplus K_N=H^*_T(M)\oplus K_N$ as $S(\mft^*)$-modules.
\item The sequence
\[
0\to H^*_T(M)\to H^*_T(M_b) \to H^*_T(M_{b+1},M_b)
\]
is exact. 
\end{enumerate}
\end{cor}
\begin{proof}
Because $H^*_T(M)\to H^*_T(M_b)$ is injective by Theorem \ref{thm:OFA+injective}, the image of $H^*_T(M)\to H^*_T(N)$ does not intersect $K_N$. That the two submodules span $H^*_T(N)$ follows directly from Proposition \ref{prop:genbythomclasses}.

To prove $(2)$, note that exactness at $H^*_T(M)$ was proven in Theorem \ref{thm:OFA+injective}. Taking $N=M_{b+1}$ in $(1)$, we see that $\im H^*_T(M)\to H^*_T(M_b)=\im H^*_T(M_{b+1})\to H^*_T(M_b)$, which is exactness at $H^*_T(M_b)$.
\end{proof}

\subsection{Actions with open-face-acyclic orbit space are Cohen-Macaulay}\label{sec:OFAisCM}

Masuda and Panov \cite[Definition 5.3]{MasudaPanov} define the {\it face ring} of $M/T$ as the graded ring
\[
\RR[M/T]=\RR [\tau_F\mid F \text{ a face of }M/T]/ I,
\]
where $I$ is the ideal generated by elements of the form
\[
\tau_F\tau_G-\tau_{F\vee G}\cdot \sum_{E\in F \cap G} \tau_E.
\]
Here, in case $F$ and $G$ have nonempty intersection, $F\vee G$ is the unique smallest face containing $F$ and $G$, and zero otherwise. The notation $E\in F\cap G$ means that $E$ is a connected component of $F\cap G$.  One proves just as in \cite[Section 6]{MasudaPanov} that the canonical (after fixing compatible orientations on all normal bundles) homomorphism $\RR[M/T]\to H^*_T(M)$ is well-defined and injective, and by Proposition \ref{prop:genbythomclasses} and Theorem \ref{thm:OFA+injective} it is surjective as well. Note that because we have already proven injectivity of $H^*_T(M)\to H^*_T(M_b)$, we can work with $H^*_T(M)$ itself instead of $H^*_T(M)/(\ker H^*_T(M)\to H^*_T(M_b))$.

\begin{ex} Consider the $T^3$-action on $S^5=\{(z_i)\mid \sum |z_i|^2=1\}\subset \CC^3$ defined by $(t_1,t_2,t_3)\cdot (z_1,z_2,z_3)=(t_1z_1,t_2z_2,t_3z_3)$ and the $T^2$-action on $\CC P^2$ defined by $(t_1,t_2)\cdot [z_0:z_1:z_2]=[z_0:t_1z_1:t_2z_2]$. These actions are Cohen-Macaulay (the latter even equivariantly formal) and their orbit spaces coincide and are open-face-acyclic. Hence the equivariant cohomologies $H^*_{T^3}(S^5)$ and $H^*_{T^2}(\CC P^2)$ are isomorphic as rings. Alternatively, this follows from Remark \ref{rem:reductiontoequivformal} as the diagonal circle $S^1$ in $T^3$ acts freely  on $S^5$ such that the induced $T^2$-action on $S^5/S^1=\CC P^2$ coincides with the action described above.
\end{ex}

\begin{lem}\label{lem:surjectiveontoface}
If $F=M^p/T$ is a face such that for every subface $G\subset F$, any intersection of a face $H$ with $G$ is connected, then $H^*_T(M)\to H^*_T(M^p)$ is surjective.
\end{lem}
\begin{proof} By Proposition \ref{prop:genbythomclasses} and Theorem \ref{thm:OFA+injective}, $H^*_T(M^p)$ is generated by the Thom classes of faces in $M^p$. Let $\tau_G$ be such a Thom class, with $G=M^q/T\subset F$. There is a unique maximal face $H=M^{\tilde{q}}/T$ in $M/T$ whose intersection with $F$ is $G$. For a $T$-invariant metric on $M$, we have that the normal bundle of $M^{\tilde{q}}$ in $M$, restricted to $M^q$, coincides with the normal bundle of $M^q$ in $M^p$. Thus, $\tau_H\in H^*_T(M)$ is mapped onto $\tau_G\in H^*_T(M^p)$. 
\end{proof}

\begin{rem} The condition on $F$ is necessary. For example, consider the action of a two-dimensional maximal torus $T\subset \mathrm{SO}(5)$ on $M=S^4$. If $K\subset T$ is a one-dimensional stabilizer, then $H^*_{T}(M)\to H^*_{T}(M^\mfk)$ is not surjective. In fact, $\dim H^2_{T}(M)=2$, but $\dim H^2_{T}(M^\mfk)=3$.
\end{rem}

\begin{thm}\label{thm:OFAisCM}
Every torus action on a compact orientable manifold with open-face-acyclic orbit space is Cohen-Macaulay.
\end{thm}
\begin{proof} We prove the theorem by induction on the dimension of the orbit space. As in \cite[Theorem 9.3]{MasudaPanov} we blow up faces of $M/T$ (i.e., we replace submanifolds $M^p$ by the complex projectivizations of the respective normal bundles $P(\nu M^p)$, see \cite[Section 9]{MasudaPanov} or \cite[Section 8]{MMP}) successively until we arrive at a $T$-manifold $\widehat{M}$ whose orbit space satisfies that the intersection of any two faces of $\widehat{M}/T$ is connected (i.e., empty or a face). In order to be able to apply Lemma \ref{lem:surjectiveontoface}, we need to choose the faces to be blown up such that the assumptions for the lemma are satisfied. More precisely, let 
\[
{\mathcal F}(M)=\{F\mid F \text{ a face of } M/T \text{ such that for all faces } G, F\cap G \text{ is connected}\}.
\]
It is clear that ${\mathcal F}(M)$ contains all vertices. If there are edges not contained in ${\mathcal F}(M)$, we can at first blow up along vertices until we obtain $M'$ such that all edges of $M'/T$ are in ${\mathcal F}(M')$. Continuing this process with the higher-dimensional faces, we obtain a sequence of $T$-manifolds $N_i$ with collapse maps 
\[
\widehat{M}=N_k\to N_{k-1}\to \ldots\to N_1 \to N_0=M
\]
such that $N_{i+1}$ is obtained from $N_i$ by blowing up a face $F_i=N_i^{p_i}/T\in {\mathcal F}(N_i)$. Note that because $F_i\in {\mathcal F}(N_i)$, every subface of the new facet $P(\nu N_i^{p_i})/T$ is in ${\mathcal F}(N_{i+1})$.

$H^*_T(\widehat{M})$ is a *local positively graded ring, with *maximal ideal generated by the homogeneous elements of positive degree. Lemma 8.2.~of \cite{MasudaPanov} implies that $\RR[\widehat{M}/T]=H^*_T(\widehat{M})$ is Cohen-Macaulay as *local graded ring. Considering $H^*_T(\widehat{M})$ as a module over itself, a graded version of \cite[Prop.~IV.12]{Serre} implies that it is also Cohen-Macaulay as an $S(\mft^*)$-module. Its Krull dimension is $r-b$, as the maximal dimension of an isotropy algebra of $\widehat{M}$ is $r-b$. It remains to show that if the action on the blown-up manifold is Cohen-Macaulay, then so is the original one. 

Assume we have already shown that $H^*_T(N_{i+1})$ is Cohen-Macaulay of Krull dimension $r-b$; we show it for $H^*_T(N_i)$. 
\[
\xymatrix@R=20pt@C=10pt{
0 \ar[r] & H^*_T(N_i,N_i^{p_i}) \ar[r] \ar[d] & H^*_T(N_i) \ar[r] \ar[d] & H^*_T(N_i^{p_i}) \ar[r] \ar[d] & 0 \\
0 \ar[r] & H^*_T(N_{i+1},P(\nu N_i^{p_i})) \ar[r] & H^*_T(N_{i+1}) \ar[r] & H^*_T(P(\nu N_i^{p_i})) \ar[r] & 0
}\]
By Lemma \ref{lem:surjectiveontoface} and the observations above, the lower horizontal sequence is exact. The left vertical map is an isomorphism, and hence the upper horizontal sequence is exact as well. We know that $H^*_T(N_{i+1})$ is Cohen-Macaulay, and $H^*_T(N_i^{p_i})$ and $H^*_T(P(\nu N_i^{p_i}))$ are Cohen-Macaulay by induction. Since all Krull dimensions are equal, it follows that $H^*_T(N_{i+1},P(\nu N_i^{p_i}))=H^*_T(N_i,N_i^{p_i})$ is either the zero module or Cohen-Macaulay of Krull dimension $r-b$ by the second statement of Lemma \ref{lem:CMseq}, and then the first statement of Lemma \ref{lem:CMseq} implies that $H^*_T(N_i)$ is Cohen-Macaulay of the same Krull dimension.
\end{proof}

\subsection{The equivariantly formal case}\label{sec:equivformalcase}

For $b=0$, i.e., $\dim M=2\dim T$, the converse of Theorem \ref{thm:OFAisCM} was proven by Bredon, see \cite[Corollary 3]{Bredon}. We thus have
\begin{thm}\label{thm:OFAequivariantlyformal}
An effective $T$-action on an orientable compact manifold $M$ with $\dim M=2\dim T$ is equivariantly formal if and only if its orbit space is open-face-acyclic.
\end{thm}

Still in the case $\dim M=2\dim T$, there is a relation between the number of connected components of $M_{i}\setminus M_{i-1}$ and the Betti numbers of $M$.  For an arbitrary $T$-action on a compact manifold $M$, Duflot calculated the Poincar\'{e} series of $H^*_T(M)$ in terms of the components of $M_i\setminus M_{i-1}$. Letting $\lambda_i$ denote the number of connected components of $M_i\setminus M_{i-1}$, her result simplifies in the case of an action with open-face-acyclic orbit space to the following
\begin{prop}[{\cite[Theorem 2]{Duflot}}] \label{prop:Poincareseries} For a $T$-action with open-face-acyclic orbit space, the Poincar\'{e} series of $H^*_T(M)$ is given by
$
\sum_{i=b}^r  \lambda_i \left(\frac{t^2}{1-t^2}\right)^{r-i}.
$ 
\end{prop}
Note that her notation is slightly different from ours: she indexes the $M_i$ by the dimensions of the isotropy groups, not the dimensions of the orbits.  Using the isomorphism between $H^*_T(M)$ and the face ring of $M/T$, Masuda and Panov \cite[Theorem 5.12, Theorem 7.7]{MasudaPanov} obtain the same equation for torus manifolds with $H^{odd}(M,\ZZ)=0$ using the fact that the Poincar\'{e} series of the face ring was determined in \cite[Proposition 3.8]{Stanley}. However, Proposition \ref{prop:Poincareseries} follows independently from this isomorphism, as we only combine our calculation of the codimensions of the $M^p$ in Corollary \ref{cor:OFAweights} with Theorem 2 of \cite{Duflot}.

On the other hand, if $T$ acts on an orientable compact manifold $M$ with open-face-acyclic orbit space and $\dim M=2\dim T=2r$, the action is equivariantly formal by Theorem \ref{thm:OFAequivariantlyformal}. Thus, we know that in this case $H^*_T(M)=H^*(M)\otimes S(\mft^*)$ as a graded $S(\mft^*)$-module, and hence the Poincar\'{e} series of $H^*_T(M)$ is given by $\frac{\sum_i b_{2i} t^{2i}}{(1-t^2)^r}$, where $b_i=\dim H^i(M)$ (note that the odd Betti numbers vanish since $H^*_T(M)$ maps injectively into $H^*_T(M^T)$, and the fixed point set consists of isolated points). Equating these two expressions for the Poincar\'{e} series, we obtain that the $\lambda_i$ determine the $b_i$ and vice versa:
\begin{prop} \label{prop:bettilambda} For a $T$-action on an orientable compact manifold $M$ with open-face-acyclic orbit space and $\dim M=2\dim T$, we have
\[
b_{2i}=b_{2(r-i)}=\sum_{j=i}^r (-1)^{j-i} {j \choose i} \lambda_j.
\]
\end{prop}
For an arbitrary equivariantly formal action, Bredon related the Poincar\'{e} series of $M$ with the Poincar\'{e} series of the (compact cohomology of the) components of $M_i\setminus M_{i-1}$ using the Atiyah-Bredon sequence, see the equation on the bottom of p.~846 in \cite{Bredon}. For a $T$-action with open-face-acyclic orbit space, his equation simplifies to Proposition \ref{prop:bettilambda}. Note that, using Poincar\'{e} duality for $M$ and the components of $(M_i\setminus M_{i-1})/T$, one can see that in the case of an equivariantly formal action on an orientable compact manifold, Theorem 2 of \cite{Duflot} is the same as the equation by Bredon.

\begin{ex}  \label{ex:figure} Consider the following $T^3$-actions on the $6$-dimensional manifolds $S^4\times S^2$ and $\CC P^3$: on $S^4\times S^2$, we regard the action given by the product of the $T^2$-action on $S^4=\{(z,w,s)\mid |z|^2+|w|^2+s^2=1\}\subset \CC^2\times \RR$ defined by
\[
(t_1,t_2)\cdot (z,w,s)=(t_1z,t_2w,s)
\]
and the standard $S^1$-action on $S^2$. On $\CC P^3$ we have the $T^3$-action given by
\[
(t_1,t_2,t_3)\cdot [z_0:z_1:z_2:z_3]=[z_0:t_1z_1:t_2z_2:t_3z_3].
\]
Both these actions are equivariantly formal (e.g.~because they have exactly $4$ fixed points, or because $H^{odd}(M)=0$ in both cases) and have open-face-acyclic orbit space. As the cohomologies of $\CC P^3$ and $S^4\times S^2$ are isomorphic as graded vector spaces, the numbers of connected components $\lambda_i$ of $M_i\setminus M_{i-1}$ have to coincide for these actions. In fact, they are $\lambda_0=4$, $\lambda_1=6$, $\lambda_2=4$ and $\lambda_3=1$. For $\CC P^3$, the orbit space is a tetrahedron, whereas for $S^4\times S^2$ it is a cylinder, see Figure \ref{fig:orbitspaces}.
\begin{figure}
\includegraphics{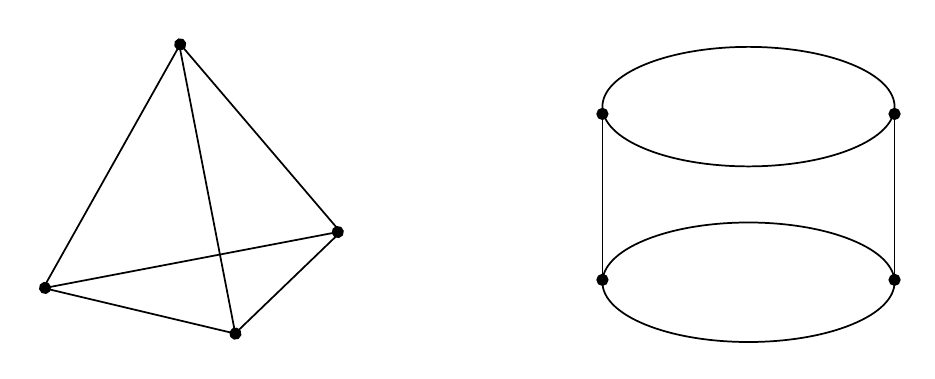}
\caption{Orbit spaces occuring in Example \ref{ex:figure}.}
\label{fig:orbitspaces}
\end{figure}
In the pictures, the dots represent the fixed points. Note that whereas the equivariant cohomologies $H^*_{T^3}(\CC P^3)$ and $H^*_{T^3}(S^4\times S^2)$ are isomorphic as graded $S(\mft^*)$-modules, their multiplicative structures do not coincide. In fact, the face rings of the tetrahedron and the cylinder are not isomorphic (e.g.~in the face ring of the cylinder the product of the Thom classes of the top and the bottom face is zero, whereas in the face ring of the tetrahedron two elements of degree two whose product is zero are linearly dependent).

Note that the tetrahedron also appears as the orbit space of e.g.~the $T^4$-action on $S^7=\{(z_i)\mid \sum |z_i|^2=1\}\subset \CC^4$ given by
\[
(t_1,t_2,t_3,t_4)\cdot (z_1,z_2,z_3,z_4)=(t_1z_1,t_2z_2,t_3z_3,t_4z_4),
\] with the vertices corresponding to the one-dimensional orbits. By Theorem \ref{thm:OFAisCM}, this action is Cohen-Macaulay, with $H^*_{T^4}(S^7)$ of Krull dimension $3$. In fact, $H^*_{T^4}(S^7)$ is isomorphic to $H^*_{T^3}(\CC P^3)$ as a ring. In view of Remark \ref{rem:reductiontoequivformal}, this is clear as the diagonal circle $S^1$ in $T^4$ acts freely  on $S^7$ such that the induced $T^3$-action on $S^7/S^1=\CC P^3$ is the action described above.
\end{ex}

\end{document}